%% file: main.tex
\documentclass{amsart}
\usepackage[utf8]{inputenc}
\usepackage{amsmath}
\usepackage{amssymb}
\usepackage{amsthm}
\usepackage[style=alphabetic,maxbibnames=8,safeinputenc,backend=biber]{biblatex}

\usepackage{mathrsfs}
\usepackage{enumerate}
\usepackage[labelfont=bf]{caption}
\usepackage{accents}
\usepackage{microtype}

\usepackage[hyperfootnotes=false]{hyperref}
\usepackage{graphicx}

\usepackage{todonotes}
\usepackage{fancyhdr}
 
\newcommand{\real}{\mathbb{R}}
\newcommand{\ints}{\mathbb{Z}}

\newcommand{\nats}{\mathbb{N}}

\newcommand{\proj}{\mathbf{P}}

\newcommand{\HH}{\mathbb{H}}
\newcommand{\Isom}{\mathrm{Isom}}
\DeclareMathOperator{\PGL}{PGL}

\DeclareMathOperator{\SL}{SL}

\newcommand{\eps}{\epsilon}
\newcommand{\del}{\partial}

\DeclareMathOperator{\id}{id}

\DeclareMathOperator{\diam}{diam}
\DeclareMathOperator{\Stab}{Stab}
\newcommand{\actson}{\curvearrowright}

\DeclareMathOperator{\Aut}{Aut}

\newcommand{\grop}[3]{\langle #2, #3 \rangle_{#1}}
\newcommand{\dirac}{\mathcal{D}}
\newcommand{\horosphere}{\mathcal{H}}
\newcommand{\cone}{\mathcal{C}}
\newcommand{\shadow}{\mathcal{O}}
\newcommand{\limshade}{\mathcal{L}}
\newcommand{\st}{\,|\,}
\newcommand{\flip}{\iota}
\newcommand{\Epsy}{\mathcal{E}}
\newcommand{\Em}{\mathcal{M}}
\newcommand{\Leb}{\mathcal{L}}
\DeclareMathOperator{\wlim}{lim*}

\theoremstyle{plain}
\newtheorem{thm}{Theorem}%[section]
\newtheorem{lem}[thm]{Lemma}
\newtheorem{prop}[thm]{Proposition}
\newtheorem{corn}[thm]{Corollary}
\newtheorem*{thm*}{Theorem}
\newtheorem*{cor}{Corollary}

\theoremstyle{definition}
\newtheorem{defn}[thm]{Definition}

\title{Ergodicity and equidistribution in strictly convex Hilbert geometry}
\author{Feng Zhu}

\begin{document}

\begin{abstract}
    We show that dynamical and counting results characteristic of negatively-curved Riemannian geometry, or more generally CAT(-1) or rank-one CAT(0) spaces, also hold for geometrically-finite strictly convex projective structures equipped with their Hilbert metric.
    
    More specifically, such structures admit a finite Sullivan measure; with respect to this measure, the Hilbert geodesic flow is strongly mixing, and orbits and primitive closed geodesics equidistribute, allowing us to asymptotically enumerate these objects.
\end{abstract}

\maketitle

In \cite{Margulis}, Margulis established counting results for uniform lattices in constant negative curvature, or equivalently for closed hyperbolic manifolds, by means of ergodicity and equidistribution results for the geodesic flows on these manifolds with respect to suitable measures. Thomas Roblin, in \cite{Roblin}, obtained analogous ergodicity and equidistribution results in the more general setting of CAT($-1$) spaces. These results include ergodicity of the horospherical foliations, mixing of the geodesic flow, orbital equidistribution of the group, equidistribution of primitive closed geodesics, and, in the geometrically finite case, asymptotic counting estimates. 
Gabriele Link later adapted similar techniques to prove similar results in the even more general setting of rank-one isometry groups of Hadamard spaces in \cite{Link}. 

These results do not directly apply to manifolds endowed with strictly convex projective structures, considered with their Hilbert metrics and associated Bowen-Margulis measures, since strictly convex Hilbert geometries are in general not CAT($-1$) or even CAT(0) (see e.g. \cite[App.\, B]{Egloff}.)
Nevertheless, these Hilbert geometries exhibit substantial similarities to Riemannian geometries of pinched negative curvature. In particular, 
there is a good theory of Busemann functions and of Patterson-Sullivan measures on these geometries.
Crampon and Marquis, in \cite{crampon_these,CM12,CM14}, used this to study the geodesic flow on these geometries and show that this flow, when restricted to the non-wandering set and under an additional ``asymptotically hyperbolic'' assumption on the cusps, this flow is uniformly hyperbolic and topologically mixing.

In this note, we show that we can obtain strong mixing of the geodesic flow with respect to a natural measure on the unit tangent bundle, the Sullivan measure (see \S\ref{sec:measures} for details), as well as some of Roblin's equidistribution results---orbital equidistribution of the group and equidistribution of primitive closed geodesics---in the setting of geometrically-finite strictly convex Hilbert geometries (``our setting'').

Briefly, $\Omega$ will be a strictly convex domain with $C^1$ boundary, i.e. an open set of $\proj(\real^{n+1})$ contained in an affine chart such that $\del\Omega$ contains no sharp points or line segments (see \S\ref{sec:hilbgeom} for a more precise description), and $\Gamma \leq \Aut(\Omega)$ will be a discrete group of projective automorphisms preserving $\Omega$. Then $\Omega / \Gamma$ is an orbifold (if $\Gamma$ is torsion-free, a manifold) equipped with a convex projective structure.

\begin{thm*}
Let $\Omega \subset \proj(\real^{n+1})$ be a strictly convex projective domain with $C^1$ boundary, and let $\Gamma \leq \Aut(\Omega)$ be a non-elementary discrete group. 
\begin{enumerate}[(i)]
\item (Theorem \ref{thm:finite_BMmeas}) If $\Gamma$ acts geometrically finitely on $\Omega$, then
$S\Omega / \Gamma$ admits a finite Sullivan measure $m_\Gamma$, associated to a $\Gamma$-equivariant conformal density $\mu$ of dimension $\delta = \delta(\Gamma)$.
\end{enumerate}
Suppose further that $S \Omega / \Gamma$ admits a finite Sullivan measure $m_\Gamma$ associated to a $\Gamma$-equivariant conformal density of dimension $\delta(\Gamma)$. Then 
\begin{enumerate}[(i)]
\setcounter{enumi}{1}
\item (Theorem \ref{thm:mixing})
the Hilbert geodesic flow $(g_\Gamma^t)_{t \in \real}$ on $S\Omega / \Gamma$ is mixing for $m_\Gamma$;

\item (Theorem \ref{thm:orbit_equidist})
for all $x, y \in \Omega$, 
\[ \delta \|m_\Gamma\| e^{-\delta t} \sum_{\substack{\gamma \in \Gamma\\ d_\Omega(x, \gamma y) \leq t}} \dirac_{\gamma y} \otimes \dirac_{\gamma^{-1} x} \] 
converges weakly in $C(\bar\Omega \times \bar\Omega)^*$ to $\mu_x \otimes \mu_y$ as $t \to \infty$;
\item (Theorem \ref{thm:pcgeod_equidist})
let $\mathcal{G}_\Gamma(\ell)$ be the set of primitive closed geodesics of length at most $\ell$. As $\ell \to +\infty$,
\[ \delta \ell e^{-\delta \ell} \sum_{g \in \mathcal{G}_\Gamma(\ell)} \dirac_g \to \frac{m_\Gamma}{\|m_\Gamma\|} \]
(a) in $C_c(S\Omega / \Gamma)^*$, and, if furthermore $\Gamma$ acts geometrically finitely on $\Omega$, also
(b) in $C_b(S\Omega / \Gamma)^*$.
\end{enumerate}
\end{thm*}
Here $d_\Omega$ denotes the Hilbert metric, $\dirac_x$ denotes a Dirac mass at $x$, and given $g \in \mathcal{G}_\Gamma(\ell)$, $\dirac_g$ denotes the normalized Lebesgue measure supported on $g$. $C(X)^*$, $C_c(X)^*$ and $C_b(X)^*$ denote the weak*-duals to, respectively, the space of continuous functions, the space of compactly-supported continuous functions, and the space of bounded continuous functions on $X$.

The proof of the mixing result follows the lines of arguments from \cite{Babillot} and \cite{Ricks}, using cross-ratios, length spectrums, and topological mixing.
Measure-theoretic mixing results such as this may be viewed as more quantitative versions of topological mixing, and can be useful tools for establishing equidistribution results such as those which follow.

The proofs of the equidistribution results follow Roblin's proofs closely,
making heavy use of mixing and of cones in the space and shadows on the boundary without reference to any notion of angle, which is not well-defined in our setting. 
These equidistribution results have as corollaries asymptotic counting results for orbits or closed geodesics, and may have applications to the study of the deformation space of strictly convex projective structures on hyperbolizable manifolds.

We remark that the result in (iv)(b) relies on structural properties of cusp regions which are more integrally linked to the geometric finiteness condition. On the other hand, the mixing and equidistribution results in (ii), (iii) and (iv)(a) apply to larger classes of discrete subgroups $\Gamma \leq \Aut(\Omega)$ admitting finite Sullivan measures; it may be interesting to explore just how large this class is.

Before proceeding with our principal contents, we will take a moment to point out further connections, involving Patterson-Sullivan theory and counting results, to the more general class of Anosov representations, defined by Labourie \cite{Labourie} and Guichard--Wienhard \cite{GW}, which may be viewed as a higher-rank analogue of convex cocompactness. 

\input{section6}

\subsection*{Organization}
The rest of the paper is organised thus: section 1 collects the necessary background on strictly convex Hilbert geometries and geometric finiteness in that setting; section 2 describes the construction and properties of Patterson-Sullivan measures and Bowen-Margulis measures for group acting geometrically finitely on such geometries. 
Sections \ref{sec:mixing}, \ref{sec:orbit_equidist} and \ref{sec:pcgeod_equidist} describe the proofs of Theorems \ref{thm:mixing}, \ref{thm:orbit_equidist}, and \ref{thm:pcgeod_equidist} respectively.

\subsection*{Acknowledgements} 
The author thanks Pierre-Louis Blayac, Harrison Bray, Fanny Kassel, and Ralf Spatzier for helpful discussions, Ludovic Marquis for supplying a citation, and Dick Canary for helpful comments on early drafts.

The author was partially supported by U.S. National Science Foundation (NSF) grant FRG 1564362, and acknowledges support from NSF grants DMS 1107452, 1107263, 1107367 ``RNMS: Geometric Structures and Representation Varieties'' (the GEAR Network).
This project has also received funding from the European Research Council (ERC) under the European Union’s Horizon 2020 research and innovation programme (ERC starting grant DiGGeS, grant agreement No. 715982).

\input{prelims}

\section{Mixing of the geodesic flow} \label{sec:mixing}
\input{mixing}

\section{Orbital equidistribution of the group} \label{sec:orbit_equidist}
\input{equidistribut}

\printbibliography

\end{document}

%% file: section6.tex
\subsection*{Counting and Patterson--Sullivan theory in higher rank} \label{sec:counting}

As alluded to above, Roblin's results continue a long line of equidistribution and counting results in negative curvature. 
Here we survey related equidistribution and counting results in the setting of word-hyperbolic subgroups of higher-rank semisimple Lie groups. In particular, this includes the setting of higher Teichm\"uller theory, which studies especially nice surface subgroups of higher-rank Lie groups.

Holonomies of strictly convex projective structures on closed 
hyperbolizable manifolds are one class of such hyperbolic subgroups; in particular, they satisfy the Anosov condition defined in \cite{Labourie} and \cite{GW}.
Holonomies of geometrically-finite strictly convex projective structures are not Anosov, but satisfy a relative version of the Anosov condition (see \cite{reldomreps} or \cite{KL}.)

In \cite{Sambarino_cvx}, Sambarino studies the class of \emph{strictly convex subgroups}, which overlaps with the class of Anosov subgroups, and includes Benoist subgroups as examples. He obtains equidistribution and counting results similar to the results presented here, but using a slightly different notion of length: 
% (the operator norm and spectral radius instead of Hilbert length):
\begin{thm*}[{\cite[Th.\,A \& B]{Sambarino_cvx}}]
Let $\Gamma < \PGL(d,\real)$ be strictly convex.
There exist $h, c > 0$ and probability measures $\mu, \bar\mu$ on $\del_\infty\Gamma$ such that, as $t \to \infty$,
\[ ce^{-ht} \sum_{\gamma \in \Gamma : \log \|\gamma\| < t} \dirac_{\gamma^-} \otimes \dirac_{\gamma^+} \to \bar\mu \otimes \mu \]
in $C(\del\Omega \times \del\Omega)^*$ and
\[ hte^{-ht}  \#\{ [\gamma] \in [\Gamma] \mbox{ primitive}: \lambda_1(\gamma) \leq t\} \to 1 \]
where $[\Gamma]$ is the set of conjugacy clases of $\Gamma$ and $\lambda_1(\gamma)$ is the logarithm of the largest eigenvalue of $\gamma$.
\end{thm*}

If $\Gamma$ is \emph{hyperconvex}, a more restricted class which includes the holonomies of strictly convex projective structures on closed {\it surfaces} as examples,
more fine-grained results \cite[Th.\,C]{Sambarino_cvx} \cite{Sambarino_hypcvx} can be obtained. 

These results are also proven using Patterson--Sullivan theory, in this case in conjunction with the thermodynamical formalism.
These methods have also been extended to the setting of special-orthogonal projective Anosov subgroups by Carvajales in \cite{LeonC} to obtain similar counting results.

%% file: prelims.tex
\section{Hilbert geometry} \label{sec:hilbgeom}
\subsection{Strictly convex Hilbert geometries}

A {\bf properly convex} domain $\Omega \subset \proj(\real^{n+1}) = \real\proj^n$ is a domain contained in some affine chart and bounded and convex in that affine chart, in the usual Euclidean sense. A properly convex domain $\Omega \subset \real\proj^n$ is {\bf strictly convex} if its boundary contains no line segments.

Given a properly convex domain $\Omega$, we can define the Hilbert metric $d_\Omega$ on it as follows: given $x, y \in \Omega$, extend the straight line between them so that it meets $\del \Omega$ in $a$ and $b$ (and $x$ is in between $a,y$.) Then $d_\Omega(x,y) = \frac 12 \log \frac{|ay||bx|}{|ax||by|}$, where $|\cdot|$ denotes Euclidean distance in the affine chart.

\begin{figure}[ht!]
    \centering
    \includegraphics[width=.32\textwidth]{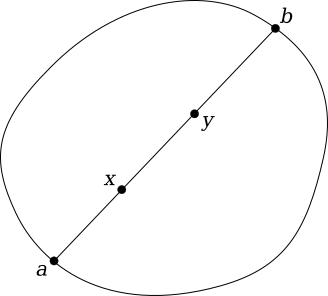}
%     \caption{} \label{fig:hilbmetr}
\end{figure}

This can be shown to be a metric, and to be projectively invariant: in particular, $d_\Omega$ is well-defined independent of the choice of affine chart.

$d_\Omega$ is a Finsler metric,
with Finsler norm at $x \in \Omega$ given by
\begin{equation}
\frac12 \left( \frac1{|xv^+|} + \frac1{|xv^-|} \right) |dx| 
\notag % \label{eqn:finsler_norm} 
\end{equation} 
where $v^\pm$ denote the forward / backward endpoints (respectively) of the geodesic through $x$ tangent to the vector $v$ whose norm we are measuring.

Straight (Euclidean) lines are always geodesics for $d_\Omega$. When $\Omega$ is {\it strictly} convex, these are the unique geodesics for this metric.

For $\Omega$ a strictly convex domain equipped with its Hilbert metric $d_\Omega$, the Hilbert geodesic flow $(g^t)_{t\in\real}$ on the unit tangent bundle $S \Omega$ is the unit (Hilbert) speed flow along the geodesics. We remark that there is a natural involution $\flip$ on this flow, given by running the flow backwards instead of forwards.

The geometry of strictly convex domains $\Omega \subset \real\proj^n$ equipped with the Hilbert metric shares many features with negative-curved Riemannian geometries, even beyond what may be expected given $\delta$-hyperbolicity, although they are not in general Riemannian or even CAT($-1$). 
For instance, nearest-point projection from a point $z$ to a Hilbert geodesic $\ell \subset \Omega$ is well-defined on the nose, not just coarsely as in the case of general $\delta$-hyperbolic space. This follows from the strict convexity of metric balls \cite[Lem.\,1.7]{CLT} and hence of the distance function $d(z,\cdot): \ell \to \real$.

Given a properly convex domain $\Omega$, we write $\Aut(\Omega)$ to denote the set of projective automorphisms preserving $\Omega$, i.e. 
\[ \Aut(\Omega) := \{T \in \PGL(\real^{n+1}) = \Aut(\real\proj^n) : T(\Omega) \subset \Omega \} .\]
Projective automorphisms $T \in \Aut(\Omega)$ are isometries of $\Omega$ equipped with the Hilbert metric $d_\Omega$; in fact, in the strictly convex case, $\Aut(\Omega)$ coincides with the isometry group of $(\Omega, d_\Omega)$ (see e.g. \cite[Prop.\,10.2]{Marquis}.)

The isometries of $(\Omega, d_\Omega)$ may be classified as hyperbolic, parabolic, or elliptic: the classification can be done in terms of translation length, or by looking at properties of matrices in $\SL(\real^{n+1})$ considered as an isomorphic image / lift of $\PGL(\real^{n+1})$. (For further details, see e.g. \cite[\S3]{CM12}.) As in the hyperbolic case, closed geodesics on a quotient $\Omega / \Gamma$ lift to axes of hyperbolic isometries on $\Omega$.

Any discrete subgroup $\Gamma \leq \Aut(\Omega)$ acts properly discontinuously on $\Omega$, with quotient $\Omega / \Gamma$ an orbifold (for $\Gamma$ torsion-free, a manifold) equipped with a {\it convex projective structure}, i.e. an atlas of charts to $\real\proj^n$ which locally give the orbifold the geometry of projective space. 
The Hilbert metric $d_\Omega$ descends to a metric on the quotient $\Omega / \Gamma$, and the Hilbert geodesic flow $(g^t)$ descends to a flow $(g_\Gamma^t)$ on the quotient $S\Omega / \Gamma $. 

For any discrete subgroup $\Gamma \leq \Aut(\Omega)$, we will have locally-finite convex fundamental domains in $\Omega$ for the action of $\Gamma$, which can be a helpful tool for studying the geometry of the quotient $\Omega / \Gamma$ and the dynamics of the geodesic flow on the quotient $S\Omega / \Gamma$:
\begin{thm}[\cite{JJLee}; {\cite[Th.\,4.5]{Marquis}}] \label{thm:cvx_lf_fundoms}
Let $\Gamma \leq \PGL(\real^{n+1})$ be a discrete subgroup acting on a properly convex open set $\Omega \subset \real\proj^n$. There exists a convex open fundamental domain $\mathcal{F}$ for the action of $\Gamma$ on $\Omega$, which is the intersection of half-spaces defined by a locally-finite family of hyperplanes.
\end{thm}
Here local finiteness means, in the first instance, that any compact subset of $\Omega$ intersects only a finite number of translates of the fundamental domain, and in the second instance that any compact subset of $\Omega$ intersects only a finite number of hyperplanes in the family. Note in particular that the boundary $\overline{\mathcal{F}} \smallsetminus \mathcal{F}$ is a locally-finite piecewise codimension-one submanifold.

We say $\Omega$ is {\bf divisible} if there is some discrete subgroup $\Gamma \leq \Aut(\Omega)$ whose action on $\Omega$ is in addition co-compact. The quotients $\Omega / \Gamma$ in this case are most closely analogous to closed hyperbolic manifolds, and share many of their good geometric and dynamical properties. In particular, in \cite{Benoist_CDI} Benoist shows that if 
$\Omega \subset \real\proj^n$ is a properly convex domain which is divisible by $\Gamma$, then the following are equivalent:
\begin{enumerate}[i)]
\item $\Omega$ is strictly convex,
\item $\del\Omega$ is $C^1$,
\item $\Gamma$ is $\delta$-hyperbolic,
\item $(\Omega, d_\Omega)$ is $\delta$-hyperbolic.
\end{enumerate}
In the course of the same work, Benoist also proved that if $\Omega$ is divisible, the geodesic flow on the quotient $\Omega / \Gamma$ is Anosov; in particular one has a splitting $TS\Omega = \real X \oplus E^s \oplus E^u$ where $X$ is the flow direction and $E^s$ and $E^u$ the stable and unstable distributions, which are respectively tangent to stable and unstable submanifolds of $S\Omega$ (which project to horospheres in $\Omega$.)

% \begin{figure}[h!]
%     \centering
%     \includegraphics[width=.35\textwidth]{figures/hilbhoros}
%     \caption{The projection of the stable and unstable manifolds to $T^1\Omega$, at a vector $v \in T^1M$ with footpoint $x$ pointing towards $x^+$: the set of tangent arrows along the horosphere based at $x^-$ is the projection of the unstable manifold; the set of arrows along the horosphere based at $x^+$ is the projection of the stable manifold.} \label{fig:hilbhoros}
% \end{figure}

\subsection{Geometric finiteness}

We will be interested in the broader class of $\Omega / \Gamma$ which are not necessarily compact, but where any non-compactness is controlled in the precise sense prescribed by geometric finiteness.

The notion of geometric finiteness arose first in the setting of Kleinian groups, and has subsequently been extended to higher dimensions and the more general setting of pinched negative curvature in \cite{Bowditch_GF}; the group-theoretic notion of relative hyperbolicity, see e.g. \cite{Bowditch_relhyp}, may be seen as an extension of geometric finiteness to a more general $\delta$-hyperbolic setting. This is essentially a notion associated with negatively-curved geometry, and generalizes the geometric and dynamical behaviour of finite-volume quotients of hyperbolic space.

Crampon and Marquis defined an analogous notion of geometric finiteness for strictly convex Hilbert geometries with $C^1$ boundary: 
\begin{defn}
Let $\Omega$ be a strictly convex domain with $C^1$ boundary and $\Gamma \leq \Aut(\Omega)$ be a discrete subgroup.

The {\bf limit set} $\Lambda_\Gamma = \Lambda_\Gamma(\Omega)$ is the subset of $\del\Omega$ given by $\overline{\Gamma \cdot x} \setminus (\Gamma \cdot x)$ for any $x \in \Omega$. $\Gamma$ is elementary if $\Lambda_\Gamma$ has at most two points; otherwise, if $\Gamma$ is non-elementary, $\Lambda_\Gamma$ is infinite and perfect.

$\xi \in \Lambda_\Gamma$ is a {\bf conical limit point} if there exists a sequence of elements $(\gamma_n) \subset \Gamma$ such that for some (and hence, by the triangle inequality, for any) $x \in \Omega$, $\gamma_n x \to \xi$ and $\sup_n d_\Omega(\gamma_n x, [x\xi)) < \infty$.

$\xi \in \Lambda_\Gamma$ is a {\bf uniformly bounded parabolic point} if the stabilizer $\Stab_\Gamma(\xi)$ acts cocompactly on $\mathcal{D}_\xi (\overline{CH(\Lambda_\Gamma \setminus \{\xi\})})$, the set of lines through $\xi$ which meet the closure of the convex hull of $\Lambda_\Gamma \setminus \{\xi\}$.
\end{defn}

We remark that our definition of conical limit points is equivalent, in this setting, to a convergence group characterization given purely in terms of the action of $\Gamma$ on $\Lambda_\Gamma \subset \del\Omega$ \cite[Lem.\,5.10]{CM12}.

\begin{defn}[\cite{CM12}]
Let $\Omega$ be a strictly convex domain with $C^1$ boundary, $\Gamma \leq \Aut(\Omega)$ be a discrete subgroup, and $M = \Omega / \Gamma$ be the corresponding quotient.

The group action $\Gamma \actson \Omega$ is {\bf geometrically finite} if every point in the limit set $\Lambda_\Gamma$ is either a conical limit point or a uniformly bounded parabolic point. In this case we also say that the quotient $\Omega / \Gamma$ is geometrically finite.
\end{defn}

Crampon and Marquis proved that this definition, given in terms of the dynamics of the group acting on orbital accumulation points in $\del\Omega$, is equivalent to several other geometric and topological descriptions, much like in the case of pinched negative curvature: \cite[Th.\,1.3]{CM12} states that for  
$\Omega$ a strictly convex domain with $C^1$ boundary, $\Gamma \leq \Aut(\Omega)$ a discrete subgroup, and $M = \Omega / \Gamma$. The following are equivalent\footnote{The acronyms are from the French: G\'eom\'etriquement Fini, Topologiquement Fini, Partie \'Epaisse du C\oe ur convexe, Partie Non cuspidale du C\oe ur convexe, Volume Fini.}:
\begin{enumerate}
\item[(GF)] $\Gamma \actson \Omega$ is geometrically finite.
\item[(TF)] The quotient $\mathcal{O}_\Gamma / \Gamma$ is an orbifold with boundary, which is the union of a compact space and a finite number of quotients of standard parabolic regions.
\item[(PEC)] The thick part of the convex core of $M$ is compact.
\item[(PNC)] The non-cuspidal part of the convex core of $M$ is compact.
\item[(VF)] The 1-neighborhood of the convex core of $M$ has finite volume, and the group $\Gamma$ is of finite-type.
\end{enumerate}
In particular, such a quotient $M$ is tame (i.e. is the interior of a compact orbifold with boundary) and hence $\Gamma$ is finitely-presented.

If, in addition, $\Omega / \Gamma$ has no cusps---or, equivalently, if the group action $\Gamma \actson \Omega$ is such that every point in the limit set $\Lambda_\Gamma$ is a conical limit point---then the action (or quotient) is {\bf convex co-compact}.

There is an analogue of Benoist's characterization of strict convexity in the finite-volume setting, due to Cooper--Long--Tillmann \cite[Th\,0.15]{CLT},
which states that for 
$M = \Omega / \Gamma$ a properly convex manifold of finite volume which is the interior of a compact manifold $N$ and where the holonomy of each component of $\del N$ is parabolic, the following are equivalent:
\begin{enumerate}[i)]
\item $\Omega$ is strictly convex,
\item $\del\Omega$ is $C^1$,
\item $\pi_1N$ is hyperbolic relative to the subgroups of the boundary components.
\end{enumerate}

Results of a similar flavor may be found in the work of Crampon--Marquis:
\begin{thm}[{\cite[Th.\,1.9]{CM12}}]
If $\Omega \subset \real\proj^n$ is strictly convex with $C^1$ boundary and $\Gamma \leq \Aut(\Omega)$ is a discrete subgroup acting geometrically finitely on $\Omega$, then 
\begin{itemize}
    % \item all of the parabolic subgroups of $\Gamma$ are conjugate into $\SO(n,1)$, 
    \item the convex hull $CH(\Lambda_\Gamma)$ with the induced Hilbert metric $d_\Omega$ is Gromov-hyperbolic, and
    \item $\Gamma$ is hyperbolic relative to the maximal parabolic subgroups (i.e. the stabilizers of parabolic points.) 
\end{itemize}
\end{thm}
We remark here that there is a fair amount of overlap between the geometric results in \cite{CM12} and \cite{CLT}, although there is some mild variation between how they choose to present these results; the reader may, in large part, consult either or both of these sources to taste.

The geodesic flow associated to a geometrically finite convex projective structure also has dynamical properties characteristic of negative curvature, at least under an additional hypothesis on the cusps:
\begin{thm}[{\cite[Th.\,5.2 \& Prop.\,6.1]{CM14}}]  \label{thm:unifhyp_topmix}
If $\Omega$ admits a geometrically finite action by some $\Gamma \leq \Aut(\Omega)$, then the Hilbert geodesic flow on $S \Omega / \Gamma$, restricted to the non-wandering set, is topologically mixing.

Moreover, if this action has with asymptotically hyperbolic cusps, then the flow, still restricted to the non-wandering set, is also uniformly hyperbolic.
\end{thm}

For a description of the ``asymptotically hyperbolic'' condition we refer the reader to \cite[\S4.3]{CM14}.
Here the non-wandering set $\mathtt{NW}$ is what Crampon and Marquis call the closed set given by the projection to $S\Omega/\Gamma$ of 
\[ \widetilde{\mathtt{NW}} := (\Lambda_\Gamma \times \Lambda_\Gamma \smallsetminus \Delta) \times \real \subset S\Omega .\]
It is non-wandering in the sense that for 
any open set $U \subset S\Omega/\Gamma$ intersecting $\mathtt{NW}$, and any $T > 0$, 
there exist $s, t > T$ such that $\phi^s U \cap U, \phi^{-t} U \cap U \neq \varnothing$.

The uniform hyperbolicity implies, in particular, that one again has, over $\widetilde{\mathtt{NW}} \subset S\Omega$, a splitting of the tangent bundle into the flow direction and the stable and unstable distributions, with the latter tangent to stable and unstable submanifolds of $S\Omega$ projecting to horospheres in $\Omega$.

Our strongest results below will apply to geometrically finite projective structures; the proofs can be simplified in the particular case of convex co-compact projective structures or divisible Hilbert geometries.

\subsection{Busemann functions and horospheres}

Suppose $\Omega$ is strictly convex with $C^1$ boundary, and $\Gamma \leq \Aut(\Omega)$.

Given $\xi \in \del\Omega$ and $x \in \Omega$, we let $c_{x,\xi}: [0,+\infty) \to \Omega$ denote the geodesic ray starting from $x$ and going towards $\xi$, i.e. $c_{x,\xi}(0) = x$ and $\lim_{t\to+\infty}c_{x,\xi}(t) = \xi$.

The {\bf Busemann function $\beta_\xi$ based at $\xi \in \del\Omega$} is a function $\beta_\xi: \Omega \times \Omega \to \real$ defined by
\[ \beta_\xi(x,y) = t - \lim_{t\to+\infty} d_\Omega(y,c_{x,\xi}(t)) = \lim_{z \to \xi} d_\Omega(x,z) - d_\Omega(y,z) .\]
We remark that this uses the sign convention adopted in \cite{Roblin}, which is a little more intuitive geometrically and helpful for working with shadows (see \S\ref{subsec:shadows}); this is opposite to the general sign convention which appears e.g. in \cite{BridsonHaefliger}.

The existence of these limits follows from the regularity assumptions on $\del\Omega$ (see e.g. \cite[\S2.2]{CM12}); these Busemann functions are $C^1$. It is immediate from the definitions that these Busemann functions satisfy a cocycle condition
\[  \beta_\xi(x, y) + \beta_\xi(y, z) =\beta_\xi(x, z) \]
for all $x,y,z \in \Omega$, and are also $\Gamma$-invariant, in the sense that
\[  \beta_{\gamma\xi}(\gamma x, \gamma y) = \beta_\xi(x,y) \]
for all $\gamma \in \Gamma$ and $x, y \in \Omega$.

The {\bf horosphere based at $\xi \in \del\Omega$} and passing through $x \in \Omega$ is the set
\[ \horosphere_\xi(x) = \{y \in \Omega : \beta_\xi(x,y) = 0\} .\]
The {\bf horoball based at $\xi \in \del\Omega$} and passing through $x \in \Omega$ is the set
\[ H_\xi(x) = \{y \in \Omega : \beta_\xi(x,y) > 0\} .\]

Horoballs are strictly convex open subdomains of $\Omega$; their boundaries are corresponding horospheres, which are $C^1$ submanifolds of $\Omega$.

If $\Gamma$ acts on $\Omega$ geometrically finitely, take $\mathcal{P}$ to be a system of representatives of $\Gamma$-orbits of parabolic fixed points of $\Gamma \actson \bar\Omega$, and, for $p \in \mathcal{P}$, let $\Pi_p$ be the (maximal parabolic) subgroup of $\Gamma$ stabilizing $p$. By geometric finiteness, $\mathcal{P}$ is finite, and for each $p \in \mathcal{P}$ we can find a horoball $H_p$ based at $p$ such that for all $\gamma \in \Gamma$, $\gamma H_p \cap H_p \neq \varnothing$ if and only if $\gamma \in \Pi_p$.

\subsection{Gromov products}

Given $x, y, z \in \Omega$, the {\bf Gromov product} $\grop{x}yz$ is defined as 
\[ \grop{x}yz = \frac 12[ d_\Omega(x,y)+d_\Omega(x,z)-d_\Omega(y,z) ] .\]
It is, roughly speaking, a measure of how much the sides of a geodesic triangle overlap. e.g. for a tree $T$, $\grop{x}yz = 0$ for any $x, y, z \in T$; more generally, the smaller the Gromov product, the thinner the geodesic triangle is. Indeed, there is a characterization of $\delta$-hyperbolicity in terms of the Gromov product.

The Gromov product can be extended to the boundary $\del\Omega$: given $\xi, \eta \in \del\Omega$ and $x \in \Omega$, define
\[ \grop{x}\xi\eta= \lim_{\substack{a_n\to\xi\\b_n\to\eta}} \grop{x}{a_n}{b_n} = \lim_{\substack{a_n\to\xi\\b_n\to\eta}} \frac 12[ d_\Omega(x,a_n)+d_\Omega(x,b_n)-d_\Omega(a_n,b_n) ] .\]
Unlike for a general $\delta$-hyperbolic space, the extension to the boundary in our setting does not require taking a $\sup \liminf$, (where the supremum is taken over all sequences $a_n\to\xi, b_n\to\eta$, see \cite[\S III.H.3.15]{BridsonHaefliger}.) Here, as in the setting of CAT($-1$) spaces, we can dispense with this by computing using geodesic triangles limiting to $\frac23$-ideal triangles with vertices $x, \xi, \eta$ (see \cite[Lem.\,5.2]{Benoist_CHQI}.)

We note the following transformation properties that are useful for proving the equivariance of our Sullivan measures below: for all $\phi \in \Isom(\Omega,d_\Omega)$,
\begin{align*}
\grop{\phi x}{\phi\xi}{\phi\eta} = \grop{x}\xi\eta & = \grop{x'}\xi\eta + \frac 12\left(\beta_\xi(x,x') + \beta_\eta(x,x') \right)
\end{align*}
We remark also that $\beta_\xi(x,u) + \beta_\eta(x,u) = 2 \grop{x}\xi\eta$, for any $u \in (\xi\eta)$. 

The following inequality will be useful below for providing estimates for our Sullivan measure:
\begin{lem} \label{lem:grop_bound}
Let $(\Omega,d_\Omega)$ be a properly convex domain with its Hilbert metric.
If a geodesic segment $(uv)$ intersects $B(x,r)$ in $\Omega$, then $\grop{x}uv \leq r$.
\begin{proof}
Suppose first that $(uv)$ is a finite geodesic segment. Pick $w \in (uv) \cap B(x,r)$. Then
\begin{align*}
d_\Omega(u, x) & \leq d_\Omega(u,w) + d_\Omega(w,x) \leq d_\Omega(u,w) + r \\
d_\Omega(v, x) & \leq d_\Omega(v,w) + r
\end{align*}
and adding the two together we get
\[ \grop{w}{x}{y} = \frac12 (d_\Omega(w,x) + d_\Omega(w,y) - d_\Omega(x, y) ) \leq r .\]
Note that these inequalities continue to hold under limits, and so the lemma continues to hold if the geodesic in question is infinite or bi-infinite.
\end{proof}
\end{lem}

\section{Patterson--Sullivan and Sullivan measures}
\label{sec:measures}
\subsection{Conformal densities and Patterson--Sullivan measures} \label{subsec:pat_sul}

Hereafter, suppose $\Omega$ is strictly convex with $C^1$ boundary, and let $\Gamma < \Aut(\Omega)$ be a discrete subgroup acting geometrically finitely on $\Omega$.

A {\bf conformal density of dimension $\delta \geq 0$} on $\Omega$ is a function $\mu$ which associates to each $x \in \Omega$ a positive finite measure $\mu_x$ on $\del \Omega$, satisfying the property that for all $x, x' \in X$, $\mu_{x'}$ is absolutely continuous with respect to $\mu_x$, with Radon-Nikodym derivative given by
\[ \frac{d\mu_{x'}}{d\mu_x}(\xi) = e^{-\delta \beta_\xi(x',x)} \]
where $\beta_\xi$ denotes the Busemann function based at $\xi \in \del\Omega$.

A density $\mu$ is said to be {\bf $\Gamma$-equivariant} if $\gamma_*\mu_x = \mu_{\gamma x}$ for all $\gamma \in \Gamma$ and all $x \in \Omega$.

The {\bf critical exponent} $\delta_\Gamma(\Omega)$ (also written $\delta_\Gamma$, $\delta(\Gamma)$ or just $\delta$ if the context is clear) of a discrete group $\Gamma$ is the critical exponent of the Poincar\'e series $g_{\Gamma,\Omega}(s,x) = g_\Gamma(s,x) := \sum_{\gamma \in \Gamma} e^{-s\cdot d_\Omega(x,\gamma x)}$, i.e. the infimum of all $s$ for which the series converges.
It is straightforward to check, using the triangle inequality, that the convergence of the Poincar\'e series, and hence the critical exponent, is well-defined independent of the choice of basepoint $x$.

The series may or may not converge at $s = \delta(\Gamma)$: if it does not, we say that $\Gamma$ is {\bf of divergence type}.

We can associate to a finitely-generated group $\Gamma$  a $\Gamma$-equivariant conformal density of dimension $\delta(\Gamma)$ as follows: fix a basepoint $o \in \Gamma$. Given $s < \delta(\Gamma)$, define 
\[ \mu_{x,s} := \frac{1}{\sum_{\gamma \in \Gamma} e^{-s \cdot d_\Omega(\gamma \cdot o, x)}} \sum_{\gamma \in \Gamma} e^{-s \cdot d_\Omega(\gamma \cdot o, x)} \dirac_{\gamma \cdot o} \]
if $\Gamma$ is of divergence type, where we use $\dirac_x$ to denote the Dirac delta measure supported at $x$, or more generally
\[ \mu_{x,s} := \frac{1}{\sum_{\gamma \in \Gamma} h(d_\Omega(\gamma \cdot o, x)) e^{-s \cdot d_\Omega(\gamma \cdot o, x)}} \sum_{\gamma \in \Gamma}  h(d_\Omega(\gamma \cdot o, x)) e^{-s \cdot d_\Omega(\gamma \cdot o, x)} \dirac_{\gamma \cdot o} \]
where $h: \real_+ \to \real_+$ is a suitable auxiliary function of subexponential growth, i.e. for any $\eta > 0$, there exists $t_\eta >0$ such that for all $s \in \real_+$ and $t > t_\eta$, $h(s+t) < e^{\eta s} h(t)$ and the modified Poincar\'e series $g'_\Gamma(s,x) = \sum_{\gamma \in \Gamma} h(d_\Omega(\gamma \cdot o, x)) e^{-s \cdot d_\Omega(\gamma \cdot o, x)}$ diverges at $s = \delta_\Gamma$ (for full details, see \cite{Patterson} or \cite{Sullivan}.)

We then take the weak* limit
\[ \mu_x = \wlim\limits_{s \searrow \delta(\Gamma)} \mu_{x,s} .\]

Following the arguments in \cite{Patterson} (see also \cite[Th.\,4.2.1]{crampon_these}), we may check that this limit is well-defined and is supported on $\del\Omega$ (indeed, on the limit set $\Lambda_\Gamma \subset \del\Omega$), that $\mu_{\gamma \cdot x} = \gamma_* \mu_x$, and that $\frac{d\mu_x}{d\mu_y}(\xi) = e^{-\delta(\Gamma) \beta_\xi(x,y)}$, as desired. 

This construction was originally due to Patterson \cite{Patterson} and Sullivan \cite{Sullivan}, and on account of this these conformal densities are known as Patterson--Sullivan densities, and the individual measures $\mu_x$ (for $x \in \del\Omega$) as Patterson--Sullivan measures.

\subsection{Shadows} \label{subsec:shadows}

Given $x, y \in \Omega $ and $r > 0$, define the shadow 
\[ \shadow_r(x,y) := \{\xi \in \del\Omega \st\, ]x\xi) \cap B(y,r) \neq \varnothing \} \]
where $]x\xi)$ denotes the geodesic ray starting from $x$ in the direction of $\xi$. We may also take $x \in \del\Omega$, in which case $]x\xi)$ should be interpreted as the bi-infinite geodesic with endpoints $x$ and $\xi$.
The terminology comes from viewing $\shadow_r(x,y)$ as the shadow cast by the ball $B(y,r)$ on the boundary $\del \Omega$, when we have a light source located at the point $x$.

A key tool for studying these shadows geometrically is the Sullivan shadow lemma, which states, informally, that the distances between orbit points may be approximated by the sizes of suitably chosen shadows. Before presenting the shadow lemma, we note a lemma used in its proof which will also be useful later on:
\begin{lem} \label{lem:roblin_12}
Let $(\Omega, d_\Omega)$ be a properly convex domain with its Hilbert metric.
For all $\xi \in \shadow_r(x,y)$, we have \[ d_\Omega(x,y) - 2r < \beta_\xi(x,y) \leq d_\Omega(x,y) .\]
\begin{proof}
See \cite[Lem.\,1.2]{Roblin}, or \cite[Lem.\,4.7]{Bray} for a version in the context of Hilbert geometry.
\end{proof}
\end{lem}

\begin{lem}[Sullivan shadow lemma] \label{lem:sullivan_shadow}
Let $\Omega$ be a strictly convex domain with $C^1$ boundary, and $\Gamma \leq \Aut(\Omega)$.
Let $\mu$ be a $\Gamma$-equivariant conformal density of dimension $\delta$ and $x \in \Omega$. Then, for all large enough $r > 0$, there exists $C > 0$ such that for all $\gamma \in \Gamma$, 
\[ \frac1C e^{-\delta \cdot d_\Omega (x,\gamma x)} \leq \mu_x\left( \shadow_r(x,\gamma x) \right) \leq Ce^{-\delta \cdot d_\Omega(x,\gamma x)} \]
\begin{proof}
% Following proof in Roblin / Bray.
% We follow the proof in \cite{Roblin}, with some added arguments for the lower bound.
Since $\gamma$ is an isometry and by $\Gamma$-equivariance,
\begin{align*} 
\mu_x(\shadow_r(x,\gamma x)) & = \mu_x(\gamma \shadow_r(\gamma^{-1} x,x)) \\
 & = \mu_{\gamma^{-1} x} (\shadow_r(\gamma^{-1} x, x)) = \int_{\shadow_r(\gamma^{-1} x, x)} e^{-\delta \cdot \beta_\xi(\gamma^{-1} x, x)} \,d\mu_x(\xi) . 
\end{align*}

From Lemma \ref{lem:roblin_12}, we have
\[ e^{-\delta \cdot d_\Omega(x,\gamma x)} \leq e^{-\delta \cdot \beta_\xi(\gamma^{-1} x, x)} \leq  e^{2\delta r} e^{-\delta \cdot d_\Omega(x,\gamma x)}  \]
and hence
\[ e^{-\delta \cdot d_\Omega(x,\gamma x)} \mu_x(\shadow_r(\gamma^{-1}x, x)) \leq \mu_x(\shadow_r(x, \gamma x)) \leq  e^{2\delta r} e^{-\delta \cdot d_\Omega(x,\gamma x)} \mu_x(\shadow_r(\gamma^{-1}x, x)) .\] 
 Since $e^{2\delta r} \mu_x(\shadow_r(\gamma^{-1}x, x)) \leq e^{2\delta r} \mu_x(\del \Omega)$, this gives us the upper bound.

To obtain the lower bound, define the function $d_x: \del\Omega \times \del\Omega \to \real_{\geq 0}$ by $d_x(\xi,\xi') = e^{-\grop{x}{\xi}{\xi'}}$ for $\xi \neq \xi'$ and $d_x(\xi,\xi) = 0$. It is clear that $d_x(\xi,\xi') = d_x(\xi',\xi)$, and $d_x(\xi,\xi') \geq 0$ with equality if and only if $\xi = \xi'$. If we knew $\Omega$ were Gromov-hyperbolic, $d_x$ would essentially be a visual metric on $\del\Omega$; more generally, it remains true that given any $\xi \in \del\Omega$, the sets 
\[ B_{d_x}(\xi,r) := \{\xi' \in \del\Omega : d_x(\xi, \xi') < r\} \] 
are nested, with $B_{d_x}(\xi,r) \searrow \{\xi\}$ as $r \searrow 0$. We call these sets ``$d_x$-balls''.

Given $A \subset \del\Omega$, write $\diam_{d_x}(A) := \sup_{\xi,\xi' \in A} d_x(\xi,\xi')$. As $r \to \infty$, the distance from $x$ to a bi-infinite geodesic joining any two points $\xi, \xi' \in \del \Omega \smallsetminus \shadow_r(y,x)$ goes to infinity, and so $\grop{x}{\xi}{\xi'} \to \infty$ for such $\xi, \xi'$. Hence $\sup_{y\in \Omega} \diam_{d_x}(\del \Omega \smallsetminus \shadow_r(y,x))\to 0$ as $r \to \infty$. 

In particular, we see that once $r$ is large enough, $\shadow_r(\gamma^{-1} x, x)$ contains a $d_x$-ball of $\mu_x$-measure at least half the measure of $\del\Omega$ minus a singleton, so 
\[ \mu_x \left( \shadow_r(\gamma^{-1}x, x) \right) \geq \frac12 \left( \|\mu_x\| - \max_{\xi\in\del X} \mu_x(\xi) \right) \]
for all sufficiently large $r$ independent of $\gamma$. Since $\Gamma$ is not elementary, $\mu_x$ is not an atom, and so this last right-hand side is some positive constant independent of $\gamma$ and $r$, giving us the lower bound.
\end{proof}
\end{lem}

We make one other observation about shadows which will be useful later:
\begin{lem} \label{lem:roblin_1B}
Suppose we have a finitely-generated subgroup $\Gamma < \Aut(\Omega)$ and fixed $x \in \Omega$ and $r>0$ such that the shadow lemma applies. 

There exists $M$ depending only on $\Gamma$, $x$ and $r$ such that for all $t \geq 1$, the family $S_t := \{\shadow_r(x, \gamma x) : \gamma\in\Gamma, t-1 < d_\Omega(x,\gamma x) \leq t \}$ covers some open subset of $\del\Omega$ with multiplicity bounded above by $M$.
\begin{proof}
Given any $\xi \in \del\Omega$, consider the geodesic ray $[x \xi)$. Since $\Gamma$ is finitely-generated, there is a uniform bound $M$ on the number of orbit points $\gamma \cdot x$ within distance $r$ of any unit-length interval of this geodesic ray.
\end{proof} \end{lem}

\subsection{(Bowen--Margulis--)Sullivan measures}
We may identify $S\Omega$ with $\del^2\Omega \times \real$, where $\del^2\Omega$ denotes the space of distinct pairs of points in $\del\Omega$, and proceed to define a measure on $S\Omega$ by
\[ dm(u) = e^{\delta_\Gamma \cdot (\beta_\xi(x,u) + \beta_\eta(x,u))} d\mu_x(\xi) \, d\mu_x(\eta) \, ds \]
for $u = (\xi,\eta,s) \in S\Omega$ (and $x \in S\Omega$ an arbitrary base-point; one can prove $dm$ is independent of the choice of $x$.)
Here $ds$ denotes (the infinitesimal form of) the Lebesgue measure $\Leb$ on the $\real$ factor; integrating against $ds$ along a single geodesic gives length according to the Hilbert metric $d_\Omega$.

We remark that $\beta_\xi(x,u) + \beta_\eta(x,u) = 2 \grop{x}\xi\eta$, where $\grop{x}\xi\eta$ denotes the Gromov product as above, and so we may also write 
\[ dm(u) = e^{2\delta_\Gamma \langle \xi,\eta \rangle_x} d\mu_x(\xi) \, d\mu_x(\eta) \, ds .\]

% By construction, this has a product structure: we may write $m = e^{\delta \beta_\xi(x,\cdot)} \mu_x \otimes e^{\delta \beta_\xi(x,\cdot)} \mu_x \otimes \Leb$ 

It is immediate that this measure is invariant under the flip involution $\flip$ on $S\Omega$; one may also check
% , via a computation, 
that it is $\Gamma$-equivariant, and hence induces a measure on the quotient $S\Omega / \Gamma$. This measure on the quotient space is really associated to the group action $\Gamma \actson \Omega$; below, we will abuse terminology slightly and also refer to this as a measure associated to $\Gamma$. We will write $m_\Gamma$ to denote the measure on the quotient $S\Omega / \Gamma$, and $m$ to denote the measure on $S\Omega$, even though both depend on $\Gamma \actson \Omega$.

We call $m$ the Sullivan measure on $S\Omega$ and $m_\Gamma$ the Sullivan measure on the quotient $S\Omega/\Gamma$, since similarly-defined measures were first studied, in the context of hyperbolic manifolds, by Sullivan \cite{Sullivan}.
For co-compact $\Gamma$, this measure coincides with the Bowen--Margulis measure, which is the unique measure of maximal entropy of a topologically mixing Anosov flow. 
% This result should hold more generally, although the details are more subtle; in any case, we will continue to call these Bowen--Margulis measures. 
We remark that Roblin in \cite{Roblin} calls these Bowen--Margulis--Sullivan measures.

The following result allows us to assume the finiteness of our Sullivan measures below. We remark that it has previously appeared in Micka\"el Crampon's thesis \cite[Th.\,4.3.1]{crampon_these} but not yet in the published literature; we therefore include a self-contained proof here for completeness, mostly following the arguments in \cite{crampon_these} but with some minor corrections and additions.
\begin{thm} \label{thm:finite_BMmeas}
If $M = \Omega / \Gamma$ is geometrically finite, then the Sullivan measure $m_\Gamma$ on $SM$ is finite.

\begin{proof}
Since the support of the Sullivan measure $m_\Gamma$ outside of the cusp neighborhoods is compact, it suffices to check that the $m_\Gamma$-measure of (the unit tangent bundle over) each cusp neighborhood is finite. 

To obtain estimates in the cusp neighborhoods, it will be useful to have the two lemmas below, the first establishing a gap between the critical exponent $\delta_\Gamma$ and the critical exponent of any parabolic subgroup, and the second showing that the Patterson--Sullivan measures have no atoms:
\begin{lem} \label{lem:critgap_para}
For any nonelementary group $\Gamma \actson \Omega \subset \real\proj^n$ containing a parabolic subgroup $P$ of rank $r$, $\delta_\Gamma > \delta_P = \frac{r}2$.
\begin{proof}
% (\cite{crampon_these}, Lemma 4.3.2.) 
% If $\Gamma$ acts on $\Omega$ and $\Omega'$ with $\Omega \subset \Omega'$, then $\delta_\Gamma(\Omega) \leq \delta_\Gamma(\Omega')$. This holds since if $x, y \in \Omega$ then $d_{\Omega'}(x,y) \leq d_\Omega(x,y)$, and so the convergence of the Poincar\'e series for $\Gamma \actson \Omega'$ implies the convergence of the Poincar\'e series for $\Gamma \actson \Omega$.

% By \cite{CM12}, Corollaire 7.18, we can find two $P$-invariant ellipsoids $\mathcal{E}^\pm$ such that $\del \mathcal{E}^\pm := \overline{\mathcal{E}^\pm} \cap \del\Omega = \{\xi_P\}$ where $\xi_P$ is the fixed point of $P$, and $\mathcal{E}^- \subset \Omega \subset \mathcal{E}^+$. $P$ acts on $\mathcal{E}^\pm$ as a parabolic subgroup of $\SO(1,n)$ acting on a horoball in hyperbolic space; by hyperbolic geometry $\delta_P(\mathcal{E}^\pm) = \frac{r}2$, and the Poincar\'e series diverges at the critical exponent. From the previous step, the same holds for $P \actson \Omega$. 
    
By \cite[Lem.\,9.8]{CM14}, $\delta_\Gamma(\Omega) \geq \delta_P(\Omega) = \frac{r}2$, and it suffices to show that the inequality is strict. 
% (\cite{crampon_these}, Lemma 4.3.4.) 
Since $\Gamma$ is nonelementary, we can use a ping-pong argument to find a free product subgroup $\langle h \rangle \times P \leq \Gamma$ where $h \in \Gamma$ is a hyperbolic element. In particular, $\Gamma$ contains all the distinct elements $g = h^{m_1} p_1 \cdots h^{m_k} p_k$ for $k \geq 1$, $n_k \in \ints_{\neq 0}$, $p_i \in P \smallsetminus \{\id\}$. Then we have a lower bound for the Poincar\'e series
\begin{align*}
    g_\Gamma(s,x) & \geq \sum_{k \geq 1} \sum_{\substack{m_1, \dots, m_k \\ p_1,\dots,p_k}} e^{-s \cdot d_\Omega(x, h^{m_1}p_1 \cdots h^{m_k}p_k x)} 
\end{align*}
and applying the triangle inequality
\[ d_\Omega(x, h^{m_1}p_1 \cdots h^{m_k}p_k x) \leq \sum_{i=1}^k d_\Omega(x,h^{m_i}x) + d_\Omega (x,p_i x) \]
to the right-hand side we obtain
\begin{align*}
    g_\Gamma(s,x) & \geq \sum_{k \geq 1} \left( \left( \sum_{n \in \ints \smallsetminus \{0\}} e^{-s \cdot d_\Omega(x, h^n x)} \right) \left( \sum_{p \in P \smallsetminus \{\id\}} e^{-s \cdot d_\Omega(x,px)} \right) \right)^k \\
    & = \sum_{k \geq 1} \left( (g_{\langle h \rangle}(s,x) - 1) (g_P(s,x) - 1) \right)^k
\end{align*}

$g_{\langle h \rangle}(s,x)$ converges for any $s > 0$, and $g_P(s,x)$ converges for any $s > \frac r2$ and diverges at $s = \frac r2$. Hence there exists $s_0 > \frac r2$ such that $(g_{\langle h \rangle}(s_0,x)-1)(g_P(s_0,x)-1) > 1$, so that $g_\Gamma(s_0,x)$ diverges. Then $\delta_\Gamma(\Omega) \geq s_0 > \frac r2$.
\end{proof} \end{lem}

\begin{prop}[\cite{crampon_these}, Proposition 4.3.5] \label{prop:noatoms}
For $\Gamma \actson \Omega$ geometrically finitely, any Patterson--Sullivan measure $\mu_x$ has no atoms. 
\begin{proof}
We can use the shadow lemma (Lemma \ref{lem:sullivan_shadow}) to show that $\mu$ has no atoms on the conical limit set. Given a conical limit point $\xi$, we have a sequence of elements $(\gamma_n^{-1}) \subset \Gamma$, a point $x \in \Omega$ and $r>0$ such that $\gamma_n^{-1} x \to \xi$ and $\gamma_n^{-1} x \in B(x_n,r)$ for some $x_n \in [x\xi)$. Thus $\xi \in \shadow_r(x,\gamma_n^{-1} x)$ for all $n$, and so 
\begin{equation} 
\mu_x(\{\xi\}) \leq \mu_x(\shadow_r(x,\gamma_n^{-1} x)) \leq C_{x,r} e^{-\delta_\Gamma d_\Omega(x, \gamma_n^{-1} x)} . 
\label{eqn:conical_shadow_approx} \end{equation}
Since $\gamma_n^{-1} \to \infty$ as $n \to \infty$ and $\delta_\Gamma > 0$, $e^{-\delta_\Gamma d_\Omega(x, \gamma_n^{-1} x)} \to 0$ as $n \to \infty$. Hence $\xi$ cannot be an atom.

It then remains to show that the measure of any of the countable number of bounded parabolic points is zero.

Let $\xi_P$ be a parabolic point and let $P$ be its stabilizer. 

We have $\mu_x(\{\xi_P\}) \leq \mu_x(V) \leq \liminf_{s \searrow \delta_\Gamma} \mu_{x,s}(V)$ for any open set $V \subset \overline\Omega$ containing $\xi$; hence it suffices to find a family of such sets $(V_n)_{n\in\nats}$ such that $\liminf_{s\searrow\delta_\Gamma} \mu_{x,s}(V_n) \to 0$ as $n \to \infty$.

Choose a convex and locally-finite open fundamental domain $\mathcal{F} \subset \Omega$ for the action of $P$ on $\Omega$ containing our basepoint $x$; such a fundamental domain exists by Theorem \ref{thm:cvx_lf_fundoms}.
Choose also a horoball $H_P$ based at $\xi_P$ that does not contain any point of the orbit $\Gamma \cdot x$.
Fix a word metric $|\cdot|$ on $P$ associated to a symmetric finite generating set $S = S^{-1}$, and let $U_n = H_P \cup \bigcup_{p\in P, |p|\geq n} p \mathcal{F}$ and $V_n = \mathrm{int}(\overline{U_n})$, where both the closure and interior are taken in $\bar\Omega$. Note $U_n \subset V_n \subset \overline{U_n}$.
% Write also $\Gamma_{\geq n} := \bigcup_{k\geq n} \Gamma_k$.
We have
\[ \mu_{x,s}(V_n) \leq \frac{1}{g'_\Gamma(s,x)} \sum_{\substack{p\in P\\|p| \geq n}} \sum_{\gamma \in \Gamma'} h(d_\Omega(x,p\gamma x)) e^{-s\cdot d_\Omega(x, p \gamma x)} \]
where $\Gamma' := \{g \in \Gamma : gx \in \overline{\mathcal{F}}\}$. Below, we will show that the contribution from each fundamental domain is on the order of $e^{-s |p|}$. This would be consistent with the results of a direct geometric computation in the case of $\Gamma = P$; in the general case,  we can control the added contributions from $\Gamma'$ and from the auxiliary function $h(\cdot)$, so that each summand in the first sum remains on the order of $e^{-s \cdot d_\Omega(x,px)}$. 

More precisely, we will show that $d_\Omega(x,p\gamma x)$ is roughly $d_\Omega(x,\gamma x) + d_\Omega(x,px)$, with bounded error controlled by a Gromov product.
By the definition of the Gromov product,
\[ d_\Omega(x, p \gamma x) = d_\Omega(x, \gamma x) + d_\Omega(x,p x) - 2 \grop{x}{\gamma x}{p^{-1} x} .\]
To bound the $\grop{x}{\gamma x}{p^{-1}x}$ term, we observe that for each $\gamma \in \Gamma'$, once $r$ and $n$ are large enough we have $V_n \subset \bigcup_{\xi \in \shadow_r(\gamma x,x)} [\gamma x, \xi)$. We can improve this to a uniform choice of $r >0$ and $n_0$ such that for all $n > n_0$,
\begin{equation} 
V_n \subset \bigcap_{\gamma \in \Gamma'}\, \bigcup_{\xi \in \shadow_r(\gamma x, x)} [\gamma x, \xi) \label{eqn:Vn_inclus}
\end{equation}
by the following argument:
by geometric finiteness, $\beta_{\xi_P}(x,y) > 0$ for all but finitely many points $y \in \Gamma' \cdot x$, i.e. all but finitely many of these points are further away from $\xi_P$ than $x$, for there is some horoball based at $\xi_P$ which does not contain any points in the orbit $\Gamma \cdot x$. Let $r_1 = \max_F d_\Omega(\gamma x, x)$ where $F$ denotes the set of $\gamma \in \Gamma'$ such that $\gamma x$ is no further from $\xi_P$ than $x$. Moreover, $\overline{\mathcal{F}} \cap \del\Omega \smallsetminus \{\xi_P\}$ and $\overline{V_n} \cap \del\Omega$ (for any $n$) are compact, so for any $n$ such that the two sets are disjoint there exists $r_2 >0$ such that any geodesic between the two intersects $B(x,r_2)$.
Let $r := \max\{2r_1,r_2\}$. 

Then $\shadow_r(\gamma x, x) \supset \shadow_{2r_1}(\gamma x, x) = \del\Omega$ for all $\gamma \in F$. For any $\gamma \in \Gamma' \smallsetminus F$, extend the geodesic ray $[\gamma x, \xi_P)$ to a bi-infinite geodesic and let $\zeta_\gamma$ denote its backwards endpoint (i.e. not $\xi_P$.) Then $\zeta_\gamma \in \overline{\mathcal{F}} \cap \del\Omega \smallsetminus \{\xi_P\}$, by our choice of $r_2$ this geodesic intersects $B(x,r) \supset B(x,r_2)$, and $\shadow_r(\zeta_\gamma,x) \subset \shadow_r(\gamma x,x)$. Hence, writing $\cap \shadow := \bigcap_{\zeta \in \overline{\mathcal{F}} \cap \del\Omega \smallsetminus \{\xi_P\}} \shadow_{r}(\zeta, x)$, we have
\[  \bigcap_{\gamma \in \Gamma'}\, \bigcup_{\xi \in \shadow_r(\gamma x, x)} [\gamma x, \xi) \supset \bigcap_{\gamma \in \Gamma' \smallsetminus F}\, \bigcup_{\xi \in \shadow_r(\zeta_\gamma, x)} [\gamma x, \xi) \supset
\bigcap_{\gamma\in\Gamma'\smallsetminus F} \bigcup_{\xi \in \bigcap\shadow} [\gamma x,\xi) \supset V_n \]
where the last inclusion holds for all sufficiently large $n$.
% note we also need 2r_2 (or, well, r_2 + epsilon) to ensure our intersections contain some open neighborhood of \xi_P in the boundary, rather than just (a priori, possibly) \xi_P itself

This implies $\grop{x}{\gamma x}{p^{-1} x} \leq r$, by the following argument together with Lemma \ref{lem:grop_bound}: by the definition of $V_n$, $p^{-1} x \in V_n$ once $|p| \geq n$; by (\ref{eqn:Vn_inclus}), this implies that the geodesic ray extending $[\gamma x, p^{-1} x)$ intersects $B(x,r)$. Since our fundamental domain $\mathcal{F}$ is convex, this geodesic ray cannot return to $\mathcal{F} \ni x$ after entering $p^{-1}\mathcal{F} \ni p^{-1} x$, and hence $[\gamma x, p^{-1} x) \cap B(x,r) \neq \varnothing$.

By this upper bound on the Gromov product, and because $h$ is an increasing function, we have
\[ \mu_{x,s}(V_n) \leq \frac{e^{2sr}}{g'_\Gamma(s,x)} \sum_{|p| \geq n} e^{-s d_\Omega(x, p x)} \sum_{\gamma\in\Gamma'} h\left(d_\Omega(x,p x) + d_\Omega(x,\gamma x)\right) e^{-s d_\Omega(x,\gamma x)} .\]

Let $\eps > 0$ and $t_\eps > 0$ be such that $\delta_\Gamma - \eps > \frac\rho2$ where $\rho$ is the highest rank of a parabolic subgroup of $\Gamma$, and for $t > t_\eps$, $h(s+t) \leq e^{\eps s} h(t)$. The number of $\gamma \in \Gamma'$ such that $d(x,\gamma x) \leq t_\eps$ is finite; let $K$ be the set of such elements, and write $L := \Gamma' \smallsetminus K$. Then
\begin{align*}
\mu_{x,s}(V_n) \leq \frac{e^{2sr}}{g'_\Gamma(s,x)} \sum_{|p|\geq n} e^{-s d_\Omega(x, p x)} \left( \Sigma_K + \Sigma_L \right)
\end{align*}
where 
\begin{align*}
\Sigma_K & := \sum_{\gamma\in K} h\left(d_\Omega(x,p x) + d_\Omega(x,\gamma x)\right) e^{-s d_\Omega(x,\gamma x)} , \\
\Sigma_L & := \sum_{\gamma\in L} h(d_\Omega(x,\gamma x)) e^{\eps d_\Omega(x, p x)} e^{-s d_\Omega(x,\gamma x)} .
\end{align*}
Because $K$ is finite, $\Sigma_K$ is bounded above independent of $s$, say by a constant $C_K$. Thus $\mu_{x,s}(V_n)$ is bounded above by
\begin{align*}
\frac{e^{2sr}}{g'_\Gamma(s,x)} \left( C_K \sum_{|p|\geq n} e^{-s d_\Omega(x, p x)} + \sum_{|p|\geq n} e^{(\eta-s) d_\Omega(x, p x)} \sum_{\gamma\in L} h(d_\Omega(x,\gamma x)) e^{-s d_\Omega(x,\gamma x)} \right)
\end{align*}
As $s \searrow \delta_\Gamma =: \delta$, $g'_\Gamma(s,x) \to \infty$ and $\sum_{p\in P} e^{-s d_\Omega(x, p x)} < \infty$. The first term thus vanishes in the limit, and we see
\begin{align*}
\mu_x(V_n) & \leq \frac{e^{2\delta r}}{g'_\Gamma(\delta,x)} \cdot \sum_{|p|\geq n} e^{-(\delta-\eta) d_\Omega(x, p x)} \sum_{\gamma\in L} h(d_\Omega(x,\gamma x)) e^{-\delta d_\Omega(x,\gamma x)}  \\
 & \leq e^{2\delta r} \mu_x(\del\Omega) \sum_{|p|\geq n} e^{-(\delta_\Gamma-\eps) d_\Omega(x, p x)}
\end{align*}
Because $\delta_\Gamma - \eps > \frac\rho2$, the series $\sum_{p \in P} e^{-(-\delta_\Gamma-\eps)d_\Omega(x,p x)}$ converges. Thus as $n \to \infty$ the right-hand side in the last inequality goes to 0, so $\mu_x(\{\xi_P\}) = 0$ as desired.
\end{proof} \end{prop}

We pause to record a corollary of the  proof which will be useful further ahead:
\begin{corn} \label{cor:conical_fullmeas}
For $\Gamma \actson \Omega$ geometrically finitely, the conical limit set has full $\mu_x$-measure for any Patterson--Sullivan measure $\mu_x$.
\end{corn}

We now use these lemmas to give a finite upper bound on the Sullivan measure of a cusp neighborhood. 
Let $\xi_P$ be a parabolic point for $\Gamma \actson \Omega$ and $P$ be its stabilizer. Let $H \subset \Omega$ be a horoball based at $\xi_P$ and $\mathcal{F}$ be a fundamental domain for $P \actson \bar\Omega$. Note $\mathcal{D} := \del \mathcal{F} \cap \Lambda_\Gamma \smallsetminus \{\xi_P\}$ is a compact fundamental domain for $P \actson \Lambda_\Gamma \smallsetminus \{\xi_P\}$. We will show that $m(S(\mathcal{F}\cap H))$ is finite. Since $m$ has no atom at $\xi_P$, this suffices to show that $m_\Gamma(SH / \Gamma) < \infty$.

From our definition of the Sullivan measure, and writing $l_\Omega$ to denote length according to the Hilbert metric, we have
\begin{align*}
m(S(\mathcal{F} \cap H)) & = \int_{\del^2\Omega} l_\Omega((\xi^- \xi^+) \cap S(\mathcal{F} \cap H)) e^{2 \delta_\Gamma \grop{x}{\xi^+}{\xi^-}} \,d\mu_x(\xi^-) \,d\mu_x(\xi^+) 
\end{align*}
We now break the right-hand side into a sum whose summands correspond to pairs of limit boundary points in different copies of $\mathcal{D}$; geometrically, this corresponds, roughly, to the different combinatorial patterns according to which bi-infinite geodesics may enter the cusp neighborhood corresponding to $P$ in the quotient: 
% This gives us 
\begin{align*}
m(S(\mathcal{F} \cap H))  & = \sum_{p,q \in P} \int_{q\mathcal{D} \times p\mathcal{D}} l_\Omega((\xi^- \xi^+) \cap S(\mathcal{F} \cap H)) e^{2 \delta_\Gamma \grop{x}{\xi^+}{\xi^-}} \,d\mu_x(\xi^-) \,d\mu_x(\xi^+) \\
 & = \sum_{p\in P} \int_{\mathcal{D} \times p\mathcal{D}} l_\Omega((\xi^- \xi^+) \cap SH) \cdot e^{2 \delta_\Gamma \grop{x}{\xi^+}{\xi^-}} \,d\mu_x(\xi^-) \,d\mu_x(\xi^+)
\end{align*}
where the last equality follows from summing over copies of the fundamental domain $\mathcal{F}$. We can now use the shadow lemma, in essence, to control the terms in this sum:

Since $\mathcal{D}$ is compact, there exists $r > 0$ such that any geodesic $(\xi^-\xi^+)$ with endpoints in $\mathcal{D}$ and $p\mathcal{D}$ and which intersects $H$ must intersect $B(x,r)$ and $B(px,r)$. In fact, since $P$ acts by isometries, we can pick a uniform such $r > 0$ over all $p \in P$.
Without loss of generality, since elements of $P$ preserve the distance to $H$, we may assume $H \cap B(x,r) \neq \varnothing$ and $H \cap B(px,r) \neq \varnothing$, and hence we conclude
\begin{enumerate}[(i)]
\item $l_\Omega((\xi^-\xi^+) \cap SH) \leq d_\Omega(x,px) + 2r$; \item $p\mathcal{D} \subset \shadow_r(x,px)$ and hence, by the shadow lemma, 
\[ \mu_x(p\mathcal{D}) \leq C_x e^{-\delta_\Gamma d_\Omega(x,px)} ;\]
\item $\grop{x}{\xi^-}{\xi^+} \leq r$ by Lemma \ref{lem:grop_bound}.
\end{enumerate}
Hence, altogether now, we have
\begin{align*}
m(S(\mathcal{F} \cap H)) & \leq \sum_{p\in P} (d_\Omega (x,px) + 2r) e^{2 \delta_\Gamma r} \mu_x(\mathcal{D}) \cdot C_x e^{-\delta_\Gamma d_\Omega(x,px)} \\
 & = e^{2\delta_\Gamma r} C_x \sum_{p\in P} (d_\Omega (x,px) + 2r) e^{-\delta_\Gamma d_\Omega(x,px)}
\end{align*}
Since $\delta_\Gamma > \delta_P$ from Lemma \ref{lem:critgap_para}, this series converges, and hence we have the desired finite upper bound for $m(S(\mathcal{F} \cap H))$. This concludes the proof of Theorem \ref{thm:finite_BMmeas}.
\end{proof}
\end{thm}

As a consequence of the proof above, we have:
\begin{prop} \label{prop:divergence}
If $\Gamma \actson \Omega$ geometrically finitely, then $\Gamma$ is of divergence type.
\begin{proof}
Given any conical limit point $\xi$, we can find a sequence of increasingly smaller shadows containing $\xi$, as described at the beginning of the proof of Proposition \ref{prop:noatoms}. Hence given any positive integer $n$ and any sufficiently large $r>0$, the conical limit set can be covered by shadows of the form $\shadow_r(x,\gamma x)$ with $|\gamma| \geq n$, which by the shadow lemma have $\mu_x$-measure bounded above by $C_{x,r} e^{-\delta_\Gamma d_\Omega(x,\gamma x)}$.

Hence, fixing $r>0$, the measure of the conical limit set is bounded above by a constant multiple of 
\[ C_x \sum_{\gamma\in\Gamma, |\gamma| \geq n} e^{-\delta_\Gamma d_\Omega(x,\gamma x)} \]
which is a tail of the (unmodified) Poincar\'e series $g_\Gamma(x,\delta_\Gamma)$.

If $\Gamma$ were of convergence type, these tails of the Poincar\'e series would go to zero as $n \to \infty$, and so taking that limit we find that the conical limit set will have zero measure in this case. 
This contradicts Corollary \ref{cor:conical_fullmeas}, which tells us that the conical limit set has full measure. 
Hence $\Gamma$ must be of divergence type.
\end{proof}
\end{prop}

This further allows us to establish ergodicity of the Hilbert geodesic flow with respect to our Sullivan measures:
\begin{defn}
Given a Borel probability space $(X,\nu)$, a flow $(g^t)_{t\in\real}$ (or $\Gamma$-action) is said to be {\bf ergodic} with respect to $\nu$ if every flow-invariant ($\Gamma$-invariant, respectively) measurable function $f: X \to \real$ is constant $\nu$-almost everywhere.
\end{defn}
Ergodicity may be viewed as a weaker form of mixing;
% \footnote{The analogy is much more apparent given von Neumann's ergodic theorem, which allows us to define ergodicity in terms of Ces\`aro convergence of space averages to time averages, whereas mixing requires weak convergence of space averages to time averages}; 
below, it will be useful for establishing mixing.
\begin{prop} \label{prop:ergodic}
$(S\Omega/\Gamma, (g_\Gamma^t)_{t\in\real}, m_\Gamma)$ and $(\del^2 \Omega, \Gamma, \mu_o \otimes \mu_o) $ are ergodic.
\begin{proof}
Recall $S\Omega = \del^2\Omega \times \real$.
There is a bijection between $\Gamma$-invariant subsets $A \subset \del^2\Omega$ and flow-invariant subsets $B \subset S\Omega / \Gamma$ given by $A \mapsto (A \times \real) / \Gamma$.
% , or in the other direction $B \mapsto \{(x,y) \in \del^2\Omega: (xy) \subset \tilde{B}\}$ where $\tilde{B} \subset S\Omega$ is the lift of $B$ to the universal cover. 
Moreover, since the measure $m = \mu_o \otimes \mu_o \otimes \mathcal{L}$ on $S\Omega$ descends to $m_\Gamma$ on the quotient $S\Omega/\Gamma$, the bijection sends sets of zero (or full) $(\mu_o \otimes \mu_o)$-measure to sets of zero (or full, respectively) $m_\Gamma$-measure.
Hence ergodicity of either one of these systems is equivalent to ergodicity of the other.

We claim that since the conical limit set $\Lambda_\Gamma^c$ has full $\mu_o$-measure (Corollary \ref{cor:conical_fullmeas}), we may deduce that $(S\Omega/\Gamma, (g_\Gamma^t)_{t\in\real}, m_\Gamma)$ is completely conservative, i.e. $S\Omega/\Gamma$ has no wandering sets of positive measure. We recall that a Borel set $U \subset S\Omega / \Gamma$ of positive measure is said to be wandering if $\int \mathbf{1}_U (g_\Gamma^t v) \,dt$ is finite for $m_\Gamma$-almost every $v \in U$. 
By the Hopf decomposition \cite[Th.\,3.2]{Krengel}, it suffices to show that any bounded positive-measure subset $V \subset S\Omega / \Gamma$ is contained in some compact set $K$ such that $\int \mathbf{1}_K (g_\Gamma^t v) \,dt = \infty$ for all $v \in V$: since $K$ is compact, it cannot contain the orbit of a positive-measure wandering set, and so $m_\Gamma$-almost every $v \in V$ must belong to the conservative part. Since this is true for any bounded positive-measure subset $V$, the dissipative part has zero measure, i.e. there are no wandering sets of positive measure, as desired.

To establish the claim, we write $S^c\Omega := (\Lambda_\Gamma^c \times \Lambda_\Gamma^c \smallsetminus \Delta) \times \real$ and note that given any $v \in S^c\Omega / \Gamma$, 
with $x \in \Omega/\Gamma$ its foot-point and $\ell_v = (v^- v^+)$ the bi-infinite geodesic in $S\Omega$ tangent to (the lift of) $v$, there exists $R_x > 0$ such that infinitely many points in the orbit $\Gamma \cdot x$ lie within distance $R_x$ of $\ell_v$ in $S\Omega$. In fact, the diameter of the compact core of $S\Omega/\Gamma$ provides an upper bound $R \geq R_x$ independent of $x \in \Omega/\Gamma$. Then, writing $B = S\bar B(x,R) \subset S\Omega/\Gamma$ we have $\int_{B} \mathbf{1}_{B}(g_\Gamma^t v) \,dt = \infty$.

Given a bounded positive-measure subset $V \subset S^c\Omega/\Gamma$ for all $x \in V$, we can find a single compact set $K \supset V$---a closure of a union $K = \overline{\bigcup_{x \in V} S\bar B(x,R)}$ of the balls just described---such that $\int \mathbf{1}_{K}(g_\Gamma^t v) \,dt = \infty$ for all $v \in V$. Since $S^c\Omega / \Gamma$ has full $m_\Gamma$-measure, we have established the desired claim.

The ergodicity of $(S\Omega/\Gamma, (g_\Gamma^t)_{t\in\real}, m_\Gamma)$ then follows from the Hopf argument, which states that for a completely conservative flow, any flow-invariant function is also invariant along the stable and unstable distributions. Since the stable and unstable distributions are tangent to globally-defined transverse horospheres in our setting, it follows from standard arguments, including the absolute continuity of the stable and unstable distributions, that a function which is invariant under the flow and also the stable and unstable distributions must be (almost everywhere) constant. For a more detailed version of this argument in the context of Hilbert geometry, see \cite[Th.\,6.7]{Bray}. 
\end{proof} \end{prop}

We remark that Propositions \ref{prop:divergence} and \ref{prop:ergodic} can also be established as part of a larger Hopf--Tsuji--Sullivan-type theorem, which establishes equivalences between several different ways of characterizing a subgroup of automorphisms as ``small'' or ``large'' in terms of conformal measures, associated Sullivan measures, and ergodicity of the geodesic flow.
Such a result was previously announced in \cite[Th.\,4.2.4]{crampon_these}, following the arguments of \cite[Th.\,1.7]{Roblin}.

We further remark that the circle of ideas that has appeared in the proof of Theorem \ref{thm:finite_BMmeas}, Proposition \ref{prop:noatoms} and their consequences are very similar to those appearing in the proof of analogous results of Dal'bo--Otal--Peign\'e in \cite{DOP}, which characterize geometrically finite Riemannian manifolds of pinched negative curvature with finite Sullivan measure in terms of Poincar\'e series.

%% file: mixing.tex
In this section, we prove our measure-theoretic mixing result. 
\begin{defn} 
Given a Borel probability space $(X,\nu)$, a flow $(g^t)_{t \in \real}$ on $X$ is said to be {\bf (strongly) mixing} with respect to $\nu$ if for all any $\varphi, \psi \in L^2(X,\nu)$,
\[ \int_X (\varphi \circ g_\Gamma^t) \cdot \psi \,d\nu \to \int_X \varphi \,d\nu \cdot \int_X \psi \,d\nu \]
as $t \to \pm\infty$, or equivalently if for all Borel subsets $A, B \subset X$, we have
\[ \nu(A \cap g^t B) \to \frac{\nu(A) \nu(B)}{\|\nu\|} .\]
as $t\to\pm\infty$
\end{defn}

Mixing is a characteristic property of geodesic flows in negative curvature. Measure-theoretic mixing results have been proven in a wide range of settings where Sullivan measures may be defined, for instance for geometrically finite subgroups in constant negative curvature (see e.g. \cite{Rudolph}), or in great generality for all discrete groups of CAT($-1$) isometries with quotient admitting a finite Sullivan measure by Roblin \cite[Th.\,3.1]{Roblin}. As far as we are aware, such results have not been announced in this context, although see \cite[\S3]{Sambarino_hypcvx} for related results about the mixing of Weyl chamber flows.

\begin{thm} \label{thm:mixing}
% [cf. \cite{Roblin}, Th\'eor\`eme 3.1] 
Let $\Omega$ be a strictly convex projective domain with $C^1$ boundary and $\Gamma \leq \Aut(\Omega)$ be a non-elementary discrete group such that $S \Omega / \Gamma$ admits a finite Bowen-Margulis measure $m_\Gamma$ associated to a $\Gamma$-equivariant conformal density of dimension $\delta(\Gamma)$.

Then the Hilbert geodesic flow $(g_\Gamma^t)_{t \in \real}$ on $S\Omega / \Gamma$ is mixing with $m_\Gamma$.
\end{thm}

We remark that we have topological mixing from \cite{CM14}, but this does not imply measure-theoretic mixing for a general dynamical system. In this case, however, it does, by arguments adapted from \cite{Babillot} and \cite{Ricks}.
We will need two lemmas. 

First, consider the {\bf length spectrum} of $(g_\Gamma^t)$, i.e. the collection of lengths of all closed geodesics in $S\Omega/\Gamma$. 
As noted in \cite[Prop.\,6.1]{CM14}, the geodesic flow is topologically transitive, and so topological mixing is equivalent to the density of the length spectrum. 
% Indeed, this is how topological mixing is proven in \cite{CM14}, by using a theorem of Benoist (which uses Zariski-density) to prove that the length spectrum must be dense in $\real$. 
In particular, we have
\begin{lem}[{\cite[Prop.\,6.1]{CM14}}] \label{lem:lenspec_dense}
For $\Omega$ and $\Gamma$ as above, the group generated by the length spectrum of $(g_\Gamma^t)$ is dense in $\real$. 
\end{lem}

We will show that if $m_\Gamma$ is not mixing, then the length spectrum is contained in a discrete subgroup of $\real$; since this is not the case, $m$ must in fact be mixing. 

Second, we have the following general lemma from ergodic theory:
\begin{lem}[{\cite[Lem.\,1]{Babillot}}] \label{lem:Babillot}
Let $(X, \mathcal{B}, \nu, (T_t)_{t \in \real})$ be a measure-preserving dynamical system, where $(X, \mathcal{B})$ is a standard Borel space, $\nu$ a Borel measure on $(X, \mathcal{B})$, and $(T_t)_{t\in \real}$ an action of $\real$ on $X$ by measure-preserving transformations. Let $\varphi \in L^2(X,\nu)$ be a real-valued function on $X$ such that $\int \varphi\,d\nu = 0$.

If there exists a sequence of reals $(t_n)$ with $t_n\to\infty$ such that $\varphi \circ T_{t_n}$ does not converge to 0 in the weak-$L^2$ topology, then there exist a sequence of reals $(s_n)$ with $s_n \to \infty$ and a non-constant function $\psi$ in $L^2(X,\nu)$ such that $\varphi\circ T_{s_n} \to \psi$ and $\varphi\circ T_{-s_n} \to \psi$ in the weak-$L^2$ topology, and hence (up to subsequence) the Ces\`aro averages
\[ A_{K^2} := \frac1{K^2} \sum_{k=1}^{K^2} \varphi\circ T_{\pm s_k} \]
converge almost surely to $\psi$.
\end{lem}

\begin{proof}[Proof of Theorem \ref{thm:mixing}]
We recall that $m_\Gamma$ is supported on the non-wandering set and is ergodic. Let $\mu$ be the ergodic current associated with $m$, i.e. $\mu$ is a measure on $\del^2\Omega$ such that $m = \mu \otimes ds$. 

Suppose the geodesic flow is not mixing with respect to $m_\Gamma$. Then there exists a compactly-supported function $\varphi \in L^2(S\Omega/\Gamma,m_\Gamma)$ on the non-wandering set such that $\int \varphi \,dm_\Gamma = 0$ but  $\varphi \circ g_\Gamma^t$ does not converge weakly to 0. By Lemma \ref{lem:Babillot} we may find a non-constant function $\psi$ which is the almost-sure limit of Ces\`aro averages of $\varphi$ for positive and negative times. Let $\tilde\psi$ be its lift to the universal cover $S\Omega$. We smooth $\tilde\psi$ along the flow by replacing it with the function $\tilde\psi_\ell(v) = \int_0^\eps \tilde\psi(\tilde{g}^s v)) \,ds$. Choosing $\eps > 0$ sufficiently small ensures that $\tilde\psi_\ell$ is not constant, and now there exists a set $E_0 \subset \del^2\Omega$ of full $\mu$-measure such that for every $v \in \pi^{-1}(E_0)$, the function $t \mapsto \tilde\psi_\ell(\tilde{g}^t v)$ is well-defined and continuous at any $t \in \real$. Concretely, $E_0 = \Lambda_\Gamma \times \Lambda_\Gamma \smallsetminus \Delta$, which consists of pairs of endpoints of geodesics in the non-wandering set.

The closed subgroup $\Pi < \real$ given by the periods of $t \mapsto \tilde\psi(\tilde{g}^t(v))$ depends only on the geodesic $(\xi,\eta)$ containing $v$, so that we get a measurable\footnote{With respect to the Borel $\sigma$-algebra for the Chabauty topology}
% to see the map is measurable, note that preimages are either countable sets (for finite intervals / corresponding to closed curves), or complements of countable sets (for infinite intervals / everything else.) 
map from $E_0$ into the set of closed subgroups of $\real$. This map is $\Gamma$-invariant, and hence by the ergodicity of $\mu$ is $\mu$-almost surely constant.

Suppose $\Pi = \real$. This would mean that $\tilde\psi$ does not depend on time, and so defines a $\Gamma$-invariant function on $E_0 = \Lambda_\Gamma^2 \setminus \Delta$. By the ergodicity of $\mu$, this function is $\mu$-almost surely constant, which contradicts that $\psi$ is not constant. Hence there must exist $a \geq 0$ such that $\Pi = a\ints$ on a set $E_1 \subset E_0$ of full $\mu$-measure.

We now use cross-ratios to conclude that if $m_\Gamma$ is not mixing, then the length spectrum is contained in the discrete subgroup $\frac a2 \ints \leq \real$.
Define the {\bf cross-ratio}
\[ B(\xi,\xi',\eta,\eta') := b_x(\xi,\eta) + b_x(\xi',\eta') - b_x(\xi,\eta') - b_x(\xi',\eta) \]
where $b_x(\xi,\eta) := \inf_{p\in X} (\beta_\xi + \beta_\eta)(p,x)$. This definition appeared in \cite{Ricks} in the setting of rank-one CAT(0) spaces; below we prove analogous properties for the cross-ratios in our setting.  Let us first show that the cross-ratio is well-defined independent of the choice of $x \in \Omega$:
\begin{lem}
% [cf. \cite{Ricks}, Lemma 5.2] 
\label{lem:Ricks52}
For any $\xi,\eta \in \del\Omega$ and $x \in \Omega$, $b_x(\xi,\eta)$ is finite, and 
\[ b_x(\xi,\eta) = (\beta_\xi + \beta_\eta)(p,x) \]
if and only if $p$ lies on the bi-infinite geodesic $(\xi\eta)$.
\begin{proof}
Given $x \in \Omega$, consider $b_{\xi\eta,x}: \Omega \to \real \cup \{-\infty\}$ defined by $y \mapsto (\beta_\xi + \beta_\eta)(y,x)$. This function is constant on $(\xi\eta)$, with value given by $d_\Omega(p_\xi x, p_\eta x)$, where $p_\xi x$ denotes where the horoball through $x$ based at $\xi$ intersects $(\xi\eta)$. We now claim that $b_{\xi\eta,x}(p_\xi y) \leq b_{\xi\eta,x}(y)$ for any $\xi,\eta$, $x$ and $y$, with equality if and only if $y \in (\xi\eta)$. To see this: note that $\beta_\xi(y,x) = \beta_\xi(p_\xi y,x)$; on the other hand, we can see geometrically that $\beta_\eta(y,x) \geq \beta_\eta(p_\xi y,x)$, with equality if and only if $y = p_\xi(y)$.
\end{proof} \end{lem}

\begin{lem}
%  [cf. \cite{Ricks}, Lemma 10.3] 
 \label{lem:Ricks103}
For any four pairwise distinct $\xi, \xi', \eta, \eta' \in \del\Omega$, we have 
\[ B(\xi,\xi',\eta,\eta') = \beta_\xi(y_0,y_1) + \beta_{\xi'}(y_2,y_3) + \beta_{\eta}(y_0,y_3) + \beta_{\eta'}(y_2,y_1) \]
for any four points $x_0 \in (\xi\eta)$, $x_1 \in (\xi\eta')$, $x_2 \in (\xi'\eta')$, and $x_3 \in (\xi'\eta)$. 
\begin{proof}
By Lemma \ref{lem:Ricks52} and the definition of the cross-ratio,
\begin{align*}
B(\xi,\xi',\eta,\eta') & = (\beta_\xi(x_0,x) + \beta_\eta(x_0,x)) + (\beta_{\xi'}(x_2,x) + \beta_{\eta'}(x_2,x)) \\
 & \phantom{ = } - (\beta_\xi(x_1,x) + \beta_{\eta'}(x_1,x)) - (\beta_\eta(x_3,x) + \beta_{\xi'}(x_3,x)).
\end{align*}
for any $x \in \Omega$, and $x_0,x_1,x_2,x_3$ as given in the statement of the lemma. The lemma then follows using the cocycle property of Busemann functions.
\end{proof} \end{lem}

Hence the cross-ratio $B(\xi,\xi',\eta,\eta')$ is well-defined independent of the choice of $x \in \Omega$, and in particular $x$ can be taken to be the same basepoint as above.

The next lemma gives an equivalent geometric definition of the cross-ratio. To state it succinctly we introduce a bit of notation: given $v \in S\Omega$ a unit tangent vector, let $\pi(v)$ denote its basepoint in $\Omega$, and let $(v^- v^+)$ be the bi-infinite geodesic it is tangent to, with backward endpoint $v^-$ and forward endpoint $v^+$. We write $\horosphere^u(v) \subset S\Omega$ % (the unstable horosphere of $v$) 
to denote the set of unit tangent vectors with basepoint on the horosphere $\horosphere _{v^-}(\pi(v))$ and tangent to geodesics with backward endpoint $v^-$, and $\horosphere^s(v) \subset S\Omega$ %(the stable horosphere of $v$) 
to denote the set of unit tangent vectors with basepoint on the horosphere $\horosphere_{v^+}(\pi(v))$ and tangent to geodesics with forward endpoint $v^+$.

\begin{lem}
% [cf. \cite{Ricks}, Lemma 10.4]
Given four pairwise distinct $\xi, \xi', \eta, \eta' \in \del\Omega$, pick $v_0 \in S\Omega$ tangent to the geodesic $(\xi\eta)$, and define, in turn, $v_1 \in \horosphere^u(v_0) \cap (\xi\eta')$, $v_2 \in \horosphere^s(v_1) \cap (\xi'\eta')$, $v_3 \in \horosphere^u(v_2) \cap (\xi'\eta)$, and $v_4 \in \horosphere^s(v_3) \cap (\xi\eta)$. 

Then $v_4 = g^{B(\xi,\xi',\eta,\eta')} v_0$ (where, recall, $g^t$ is the Hilbert geodesic flow.)
\begin{proof}
From Lemma \ref{lem:Ricks103}, we have, writing $x_i := \pi(v_i)$ to denote the basepoints of our vectors,
\[ B(\xi,\xi',\eta,\eta') = \beta_\xi(x_0,x_1) + \beta_{\xi'}(x_2,x_3) + \beta_\eta(x_0,x_3) + \beta_{\eta'}(x_2,x_1) .\]
By our choice of $v_1, v_2, v_3, v_4$, we have $\beta_\xi(x_0,x_1) = \beta_{\xi'}(x_2,x_3) = \beta_\eta(x_3,x_4) = \beta_{\eta'}(x_2,x_1) = 0$. Hence
\[ B(\xi,\xi',\eta,\eta') = \beta_\eta(x_0,x_3) = \beta_\eta(x_0,x_4) \]
by the cocycle properties of our Busemann functions. Since $v_4$ and $v_0$ are both tangent to $(\xi\eta)$, pointing towards $\eta$, $v_4 = g^{\beta_\eta(x_0,x_4)} = g^{B(\xi,\xi',\eta,\eta')} v_0$. \end{proof}
\end{lem}

Thus for any suitable 4-tuple $(\xi,\xi',\eta,\eta')$ of pairwise distinct points in $\del\Omega$, $B(\xi,\xi',\eta,\eta')$ is a period of $t \mapsto \tilde\psi_\ell(\tilde{g}^t v)$ and thus belongs to our above closed group $a\ints$. 

To be precise, ``suitable'' here means  the following: $\tilde\psi_\ell$ is the almost-sure limit of the corresponding smoothed Ces\`aro averages of $\tilde\varphi$, so if $\tilde\psi^+$ and $\tilde\psi^-$ are the upper limits of these averages for positive / negative times (resp.), the set
\[ E = \{(\xi,\eta) \in E_1 \st \tilde\psi^+(v) = \tilde\psi^-(v) = \tilde\psi(v) \,\forall v \in \pi^{-1}(\xi,\eta) \} \]
has full $\mu$-measure. 
% By the uniform hyperbolicity of the geodesic flow on the non-wandering set (Theorem \ref{thm:unifhyp_topmix}) and the uniform continuity of $\varphi$, $\tilde\psi^+$ is constant along any stable leaf, and $\tilde\psi^-$ constant along any unstable leaf. 
We use the product structure of $\mu$ to define
\begin{align*}
E^- & := \{\xi \in \Lambda_\Gamma \st (\xi,\eta') \in E \, \mu_x\mbox{-almost every }\eta'\} \\
E^+ & := \{\eta \in \Lambda_\Gamma \st (\xi',\eta) \in E \, \mu_x\mbox{-almost every }\xi'\}
\end{align*}
By Fubini's theorem, we have $\mu_x(E^-) = \mu_x(E^+) = 1$ and so $E^- \times E^+$ has full measure.

We say then that $(\xi, \eta, \xi', \eta') \in \del^4\Omega$ is a suitable tuple if $(\xi,\eta) \in E \cap (E^- \times E^+)$ and $(\xi',\eta), (\xi,\eta'), (\xi',\eta') \in E$. 

Note that the subset $E_4$ of suitable tuples has full measure --- $E \cap (E^- \times E^+)$ has full $\mu$-measure, and for any $(\xi,\eta)$ in this set, the set of $(\xi',\eta')$ satisfying the second condition has full $\mu$-measure. By continuity of the cross-ratio, the conclusion that $B(\xi,\eta,\xi',\eta') \in a\ints$ extends to all 4-tuples of pairwise distinct points $(\xi,\xi',\eta,\eta') \in \del^4\Omega$.

Finally, we relate cross-ratios to hyperbolic translation lengths:
\begin{lem}
% [cf. \cite{Ricks}, Lemma 10.6]
If $\gamma$ is a hyperbolic isometry of $\Gamma$, then $B(\gamma^-, \gamma^+, \gamma\xi, \xi) = 2\ell(\gamma)$ for all $\xi \in \del\Omega$. 
\begin{proof}
% By Lemma \ref{lem:Ricks52}, $\beta_x(\xi,\eta) = (\beta_\xi + \beta_\eta)(p,x)$ for any $p \in (\eta\xi)$. 
By Lemma \ref{lem:Ricks103}, we have
\[ B(\gamma^-, \gamma^+, \gamma\xi, \xi) = \beta_{\gamma^-}(w,x) + \beta_{\gamma^+}(y,z) + \beta_{\gamma\xi}(w,z) + \beta_\xi(y,x) \]
where, writing $\zeta := \gamma\xi$, $w \in (\gamma^-\zeta)$, $x \in (\gamma^- \xi)$, $y \in (\gamma^+ \xi)$ and $z \in (\gamma^+\zeta)$. We observe that the geodesics $\gamma^-\zeta$ and $\gamma^+\xi$ necessarily intersect by the north-south dynamics of $\gamma$, so we may take $w = y = (\gamma^-\zeta) \cap (\gamma^+ \xi)$, and thus
\[ B(\gamma^-, \gamma^+, \gamma\xi, \xi) = \beta_{\gamma^-}(y,x) + \beta_{\gamma^+}(y,z) + \beta_{\gamma\xi}(y,z) + \beta_\xi(y,x) .\]
Observe that $\gamma$ sends $(\gamma^+\xi)$ to $(\gamma^+\zeta)$ and $(\gamma^-\xi)$ to $(\gamma^-\zeta)$, so we may choose $z = \gamma y$ and $x = \gamma^{-1} y$.  
% \footnote{These lemmas from \cite{Ricks} being cited here are formulated in the setting of CAT(0) spaces with actions by groups with a rank-one element, but their proofs, which mostly involve manipulations of Busemann functions, carry over word-for-word (or estimate-by-estimate) to our setting.} 
Then, using the invariance of the Busemann functions, we have
\[ B(\gamma^-, \gamma^+, \gamma\xi, \xi) = \beta_{\gamma^-}(y,\gamma^{-1}y) + \beta_{\gamma^+}(y,\gamma y) + \beta_{\gamma\xi}(\gamma y,y) + \beta_{\gamma\xi}(y, \gamma y) .\]
Using the cocycle properties and the fact that $\beta_{\gamma^+}(y, \gamma y) = \ell(\gamma)$ for any hyperbolic isometry $\gamma$, we are done.
\end{proof} \end{lem}

Hence the length spectrum would be contained in the closed subgroup $\frac a2 \ints$. Since this would contradict Lemma \ref{lem:lenspec_dense} in our case, we conclude that the geodesic flow must in fact be mixing with respect to $m$.
\end{proof}

We note that Babillot obtains equidistribution of the horospheres as a consequence of mixing of the geodesic flow \cite[Th.\,3]{Babillot}, 
and we can do likewise here.

%% file: equidistribut.tex
In this section we prove an orbital equidistribution result, with consequences for orbital counting functions:
\newcommand{\convexhull}{\mathcal{CH}}
\begin{thm}[cf. {\cite[Th.\,4.1.1]{Roblin}}] \label{thm:orbit_equidist}
Suppose $\Omega$ is a strictly convex projective domain with $C^1$ boundary and  
$\Gamma \leq \Aut(\Omega)$ is a non-elementary discrete subgroup such that $S \Omega / \Gamma$ admits a finite Sullivan measure $m_\Gamma$ associated to a $\Gamma$-equivariant conformal density $\mu$ of dimension $\delta(\Gamma)$.

Then, for all $x, y \in \Omega$,  
\[ \delta \|m_\Gamma\| e^{-\delta t} \sum_{\substack{\gamma \in \Gamma\\ d_\Omega(x, \gamma y) \leq t}} \dirac_{\gamma y} \otimes \dirac_{\gamma^{-1} x} \] 
converges weakly in $C(\bar\Omega \times \bar\Omega)^*$ to $\mu_x \otimes \mu_y$ as $t \to \infty$.
\end{thm}
% \begin{rmk}
% This result can be extended to all $x, y \in \Omega$. We have restricted our attention here to $x, y \in \convexhull$ as it removes some of the additional technicalities in the proof without detracting from the essence of the proof or the result. 
% \end{rmk}

This has as immediate corollaries, by integrating in one or both factors,
\begin{cor}
$\delta \|m_\Gamma\| e^{-\delta t} \sum\limits_{\substack{\gamma \in \Gamma \\ d_\Omega(x,\gamma y) \leq t}}  \dirac_{\gamma y} \to \|\mu_y\| \, \mu_x$ weakly in $C(\bar\Omega)^*$. 
\end{cor}
\begin{cor}
$\# \left\{ \gamma \in \Gamma \st d_\Omega(x, \gamma y) \leq t \right\} \sim \frac{\|\mu_x\| \|\mu_y\|}{\delta \|m_\Gamma\|} e^{\delta t}$, i.e. the ratio of the two sides goes to 1 as $t \to \infty$.
\end{cor}

The second corollary is most directly a counting result; the corollary before that 
is an equidistribution result for the Patterson-Sullivan measures. The theorem includes both of these statements, and is more directly related to the mixing of the Hilbert geodesic flow (Theorem \ref{thm:mixing}), as we shall see below in the proof.

These are very much in the spirit of results first formulated by Margulis in the setting of closed manifolds of constant negative curvature \cite{Margulis} and subsequently extended and generalized to much more general settings with some trace of negative curvature. We refer the interested reader to the beginning of \cite[Ch.\,4]{Roblin} for a more extended discussion of this history.

\begin{proof}[Proof of Theorem \ref{thm:orbit_equidist}]
The proof follows that of \cite[Th.\,4.1.1]{Roblin}, with minor corrections as noted in \cite[\S6 \& \S8]{Link}. We give a brief presentation of the proof here for completeness.

Let $\nu^t_{x,y}$ denote the measure $\delta \|m_\Gamma\| e^{-\delta t} \sum_{\gamma \in \Gamma} \dirac_{\gamma y} \otimes \dirac_{\gamma^{-1} x}$. To prove the desired convergence, we need to show that 
\[ \int_{\bar\Omega \times \bar\Omega} \varphi \,d\nu^t_{x,y} \to \int_{\bar\Omega \times \bar\Omega} \varphi \,d(\mu_x \otimes \mu_y) \]
as $t\to\infty$ for all $\varphi \in C(\bar\Omega \times \bar\Omega)$.

Let us give a overview of the proof before plunging into some of the details. The proof uses mixing of the geodesic flow applied to suitable geometrically-described sets: given $x \in \Omega$, $A \subset \del\Omega$, and $r>0$, define
\begin{align*}
\cone_r^+(x,A) & := \left\{ y \in \Omega \st \exists x' \in B(x,r), \xi \in A : B(y, r) \cap \,]x'\xi) \neq \varnothing \right\} \\
\cone_r^-(x,A) & := \left\{ y \in \Omega \st B(y,r) \subset \bigcap_{x' \in B(x,r)} \; \bigcup_{\xi\in A} \; ]x'\xi) \right\} 
\end{align*}
These may be thought of as expanded or contracted cones from $x$ to $A$, with the parameter $r$ controlling the expansion or contraction. We can use mixing to show that the $(\mu_x \otimes \mu_y)$-measures of sufficiently small $\bar{\mathcal{A}} \times \bar{\mathcal{B}} \subset \bar\Omega \times \bar\Omega$ are uniformly well-approximated by $\nu_{x,y}^t$-measures of corresponding products of cones over $A$ and $B$. Here ``sufficiently small'' means ``contained in one of a system of neighborhoods $\hat{V} \times \hat{W} \subset \bar\Omega \times \bar\Omega$, one for each $(\xi_0,\eta_0) \in \del\Omega \times \del\Omega$.'' 

We can then approximate, topologically and hence in measure, any sufficiently small Borel subset of $\bar\Omega \times \bar\Omega$ by products of cones. From there, using standard arguments to approximate continuous positive functions using characteristic functions, we obtain the desired convergence of integrals if we replace the domain $\bar\Omega \times \bar\Omega$ with $\hat{V} \times \hat{W}$. We are then done by taking a finite subcover of the cover of the compact $\bar\Omega \times \bar\Omega$ by these neighborhoods $\hat{V} \times \hat{W}$ and using a partition of unity subordinate to this subcover.

The technical crux of the proof is then the following
\begin{prop} \label{prop:1ere_etage}
Fix $\eps > 0$, $(\xi_0, \eta_0) \in \del\Omega \times \del\Omega$ and $x, y \in \Omega$. Then there exist $r>0$ and open neighborhoods $V$ and $W$ of $\xi_0$ and $\eta_0$ (resp.) in $\del\Omega$, such that for all Borel subsets $A \subset V, B \subset W$, we have, as $T \to +\infty$, 
\begin{align*}
\limsup \nu_{x,y}^T (\cone_r^-(x,A) \times \cone_r^-(y,B)) & \leq e^\eps \mu_x(A) \mu_y(B) \\
\liminf \nu_{x,y}^T (\cone_r^+(x,A) \times \cone_r^+(y,B)) & \geq e^{-\eps} \mu_x(A) \mu_y(B)
\end{align*}

\begin{proof}
The proof estimates the $\nu_{x,y}^T$-measure of the product of cones using a number of other geometric objects, all naturally equivariant under the isometries of $\Omega$: 
\begin{enumerate}[1.]
\item For $z \in \Omega$ and $(\xi,\eta) \in \del^2\Omega$, let $z_{\xi\eta}$ denote the point of $S\Omega$ parallel to $(\eta\xi)$ {\footnotesize (i.e. determining a geodesic with forward endpoint $\xi$)} with foot-point the nearest-point projection of $z$ onto $(\xi\eta)$. Given in addition $r > 0$ and $A \subset \del\Omega$, define 
\[ K^+(z,r,A) = \left\{ g^s z_{\xi\eta} \st -\frac{r}2 < s < \frac{r}2, (\xi,\eta) \in \del^2\Omega, \eta \in A, d_\Omega(z,(\xi\eta)) < r \right\} .\]
% Roughly, $K^+(z,r,A)$ is a smoothed set of points in $S\Omega$ which determine geodesics with forward endpoint in $A$ and which have foot-point $r$-close to $z$.

Inverting the role of $\xi$ and $\eta$ in the above definition yields $\iota K^+(z,r,A) =: K^-(z,r,A)$. We will also write $K(z,r)$ to denote $K^+(z,r,\del\Omega) \cup K^-(z,r,\del\Omega)$. We remark that $K(z,r) \subset SB(z,3r/2)$ by construction.

\item Given $r > 0$ and $a, b \in \Omega$ with $d_\Omega(a,b) > 2r$, we will consider the following enlarged and contracted shadows:
\begin{align*}
\shadow_r^+(a,b) & := \left\{ \xi \in \del\Omega \st \exists a' \in B(a,r):\, ]a'\xi) \cap B(b,r) \neq \varnothing \right\} \\
\shadow_r^-(a,b) & := \left\{ \xi \in \del\Omega \st \forall a' \in B(a,r):\, ]a'\xi) \cap B(b,r) \neq \varnothing \right\}
\end{align*}
When $a \to \eta \in \del \Omega$, these variant shadows have as a common limit
\[ \shadow_r(\eta,b) = \left\{\xi \in \del\Omega \st (\eta\xi) \cap B(b,r) \neq \varnothing \right\} =: \shadow_r^\pm(\eta,b) .\]

\item For $r > 0$ and $a, b \in \Omega$ with $d_\Omega(a,b) > 2r$, we denote by $\limshade_r(a,b)$ the set of $(\xi,\eta) \in \del^2 \Omega$ such that the geodesic $(\xi\eta)$ crosses first $B(a,r)$ and then $B(b,r)$. Observe that
\[ \limshade_r(a,b) \subset \shadow_r^+(b,a) \times \shadow_r^+(a,b) .\]
On the other hand, as noted in \cite{Link}, it is in general not true, even when $\Omega = \HH^2$, that \[ \shadow_r^-(b,a) \times \shadow_r^-(a,b) \subset \limshade_r(a,b) ,\]
although we do have the following
\begin{lem}[{\cite[Lem.\,6.6]{Link}}] \label{lem:Link66}
If $(\gamma y, \gamma^{-1} x) \in \cone_r^-(x,A) \times \cone_r^-(y,B)$, then
\[ \left\{ (\zeta,\xi) \in \del\Omega \times \del\Omega : \xi \in \shadow_r^-(x,\gamma y), \zeta \in \shadow_r(\xi, x) \right\} \subset \limshade_r(x, \gamma y) \cap (\gamma B \times A) .\]
\begin{proof}
This follows the proof in \cite{Link} once we establish the following claim: if $\gamma y \in \cone_r^-(x,A)$, then $\shadow_r^+(x,\gamma y) \subset A$. To see this, we note that by our (Roblin's) definition of $\cone_r^-(x,A)$, given any point $z \in B(\gamma y,r)$ {\it and any point} $w \in B(x,r)$, $z$ is on some geodesic ray starting at $w$ with endpoint (i.e. asymptotic to some point) in $A$. By the uniqueness of geodesics in our strictly convex setting, any geodesic ray starting in $B(x,r)$ and passing through $B(\gamma y,r)$ has endpoint in $A$. On the other hand, $\shadow_r^-(x,\gamma y)$ consists of all the points $\xi \in \del\Omega$ such that for every point $w \in B(x,r)$, there exists a geodesic ray starting at $w$ and intersecting $B(\gamma y,r)$ with endpoint $\xi$. As noted above, this implies $\xi \in A$. 
\end{proof} \end{lem}
\end{enumerate}

Now choose $r \in (0, \min\{1,\eps/30\delta\} )$ with $\mu_x(\del\shadow_r(\xi_0,x)) = 0 = \mu_y(\del\shadow_r(\eta_0,y))$ This is possible since for any $\xi \in \del\Omega, x \in x$, $\del\shadow_{r_1}(\xi,x) \subset \shadow_{r_0}(\xi,x)$ for $r_1 < r_0$, and $\mu(\shadow_{r_0}(\xi,x)) \to 0$ as $r_0 \to 0$ by the shadow lemma.

{\bf We first prove the result for $x,y \in \Omega$ where $x \in(\xi_0 \xi_0')$ and $y \in (\eta_0 \eta_0')$ where $\xi_0', \eta_0' \in \Lambda_\Gamma$.} Thus we have $\xi_0' \in \shadow_r(\xi_0,x)$, and similarly $\eta_0' \in \shadow_r(\eta_0,y)$; hence
\[ \mu_x(\shadow_r(\xi_0,x)) \mu_y(\shadow_r(\eta_0,y)) > 0 .\]
Take two open sets $\hat{V}, \hat{W}$ of $\bar\Omega$, containing $\xi_0, \eta_0$ respectively, and sufficiently small so that for all $a \in \hat{V}, b \in \hat{W}$, we have 
\begin{align*}
e^{-\eps/30} \mu_x(\shadow_r(\xi_0,x)) \leq \mu_x(\shadow_r^\pm(a,x)) & \leq e^{\eps/30} \mu_x (\shadow_r(\xi_0,x)) \\
e^{-\eps/30} \mu_y(\shadow_r(\eta_0,y)) \leq \mu_y(\shadow_r^\pm(b,y)) & \leq e^{\eps/30} \mu_y (\shadow_r(\eta_0,y))
\end{align*}
Choose open neighborhoods $V, W$ of $\xi_0, \eta_0$ (resp.) in $\del\Omega$, such that $\bar{V} \subset \hat{V} \cap \del\Omega$ and $\bar{W} \subset \hat{W} \cap \del\Omega$. These will be the open neighborhoods we desire.

Given Borel subsets $A \subset V$, $B \subset W$, write $K^+ = K^+(x,r,A)$ and $K^- = K^-(y,r,B)$. 
The proof proceeds by estimating asymptotically (as $T\to+\infty$) the quantity
\[ \int_0^T e^{\delta t} \sum_{\gamma \in \Gamma} m(K^+ \cap g^{-t} \gamma K^-) \,dt \]
by examining how the elements of $\Gamma$ contribute to the various parts of it; the elements which do contribute here are (up to uniformly bounded error) exactly those we are seeking to enumerate in this stage, that is elements $\gamma \in \Gamma$ such that $d_\Omega(x,\gamma y) \leq T$, $\gamma y \in \cone_1^\pm(x,A)$ and $\gamma^{-1} x \in \cone_1^\pm(y,B)$.

To see this last claim: we may verify, by recalling the definitions of $m$ and $K^\pm$, that for $\gamma \in \Gamma$ with $d_\Omega(x, \gamma y) > 2r$ we have
\[ m(K^+ \cap g^{-t} \gamma K^-) = \int \frac{d\mu_x(\xi) \,d\mu_x(\eta)}{e^{-2 \delta \grop{x}\xi\eta}} \int_{-r/2}^{r/2} \mathbf{1}_{K(\gamma y,r)} (g^{t+s} x_{\xi\eta} ) \,ds \]
where the integral is supported on $\limshade_r(x, \gamma y) \cap (\gamma B \times A)$, and note that 
\[ \int_0^{T-3r} e^{\delta t} \mathbf{1}_{K(\gamma y,r)} (g^{t+s} x_{\xi\eta} ) dt = 0 \]
if $d_\Omega(x,\gamma y) > T$.  With some work we have that, for the case where $d_\Omega(x, \gamma y) > 2r$,
\[ \int \frac{d\mu_x(\xi) \,d\mu_x(\eta)}{e^{-2 \delta \grop{x}\xi\eta}} \int_{-r/2}^{r/2} \mathbf{1}_{K(\gamma y,r)} (g^{t+s} x_{\xi\eta} ) \,ds \sim r^2 \mu_x(\shadow_r(\xi_0,x)) \mu_y(\eta_0,y)) .\]

From the strong mixing of the geodesic flow in the quotient (Theorem \ref{thm:mixing}), we have, for large enough $t$, 
\begin{equation*} e^{-\eps/3} m(K^+) m(K^-) \leq \|m_\Gamma\| \sum_{\gamma \in \Gamma} m(K^+ \cap g^{-t} \gamma K^-) \leq e^{\eps/5} m(K^+) m(K^-) . 
% \label{eqn:411_mixing} 
\end{equation*}
But now we can use the definition of $K^\pm$ to find that $m(K^\pm)$ is bounded between constant multiples of $\mu_x(A)$ or $\mu_y(B)$, with the bounding constants given in terms of shadows: for $K^+ = K(x,r,A)$ we have
\[ m(K^+) = r \int_A d\mu_x(\xi) \int_{\shadow_r(\xi,x)} e^{\delta \grop{x}\xi\zeta} d\mu_x(\zeta) .\]
Since $\grop{x}\xi\zeta \leq r$ (from Lemma \ref{lem:grop_bound}) and $A \subset V$, we obtain
\begin{equation*} e^{-\eps/30} r \mu_x(A) \mu_x(\shadow_r(\xi_0,x)) \leq m(K^+) \leq e^{\eps/10} r \mu_x(A) \mu_x(\shadow_r(\xi_0,x)) .\label{eqn:411_K+} \end{equation*}
Arguing similarly with $K^- = K^-(y,r,B)$, we obtain
\begin{equation*} e^{-\eps/30} r \mu_y(B) \mu_y(\shadow_r(\eta_0,y)) \leq m(K^-) \leq e^{\eps/10} r \mu_y(B) \mu_y(\shadow_r(\eta_0,y)) . \label{eqn:411_K-} \end{equation*}

Thus, to outline, we have, for all sufficiently large $T$,
\[ \int_0^T e^{\delta t} \sum_{\gamma \in \Gamma} m(K^+ \cap g^{-t} \gamma K^-) \,dt \sim \frac{M}{\delta \|m_\Gamma\|} e^{\delta T} \mu_x(A) \mu_y(B) \]
from mixing on the one hand, where $M:= r^2 \mu_x(\shadow_r(\xi_0,x)) \mu_y(\shadow_r(\eta_0,y))$, and 
\[ \int_0^T e^{\delta t} \sum_{\gamma \in \Gamma} m(K^+ \cap g^{-t} \gamma K^-) \,dt \sim \frac{M}{\delta \|m_\Gamma\|} e^{\delta T} \nu_{x,y}^T(\cone_1^\pm(x,A) \times \cone_1^\pm(y,B)) \]
from our earlier arguments based in geometric considerations on the other. 
Modulo (many!) careful details, for which we refer the interested reader to the proof of \cite[Th.\,4.1.1]{Roblin} (Premi\`ere \'etape), or \cite[\S8]{Link} for the amended proof of the lower bound, this establishes the lemma (with $r=1$) in the case when $x \in(\xi_0 \xi_0')$ and $y \in (\eta_0 \eta_0')$ where $\xi_0', \eta_0' \in \Lambda_\Gamma$.

{\bf For more general $x,y \in \Omega$, not satisfying this restriction relative to $\xi_0,\eta_0$}, we can reduce to the previous case as follows: choose $\zeta_0 \in \Lambda_\Gamma \smallsetminus \{\xi_0,\eta_0\}$, and $x_0\in (\xi_0\zeta_0)$ and $y_0 \in (\eta_0\zeta_0)$. 

From the previous step we have neighborhoods $V_0, W_0$ of $\xi_0, \eta_0$ (respectively) such that the result of the lemma holds for $x_0$ and $y_0$ in the place of $x$ and $y$ and $V_0$ and $W_0$ in the place of $V$ and $W$; we can then relate the orbit of $y$ seen from $x$ to that of $y_0$ seen from $x_0$, and expanding $r$ slightly this will will again establish the lemma. For the details of this step, we refer the interested reader to the proof of \cite[Th.\,4.1.1]{Roblin} (Deuxi\`eme \'etape).
\end{proof} \end{prop}

To complete the proof of the theorem: let $x, y \in \Omega$ and fix $\eps > 0$.
Let $V$ and $W$ be the open sets from Lemma \ref{prop:1ere_etage} for some $(\xi_0,\eta_0) \in \del\Omega \times \del\Omega$, and take two open sets $\hat{V}, \hat{W}$ of $\bar\Omega$ such that $\hat{V} \cap \del\Omega = V$, $\hat{W} \cap \del\Omega = W$, and consider two Borel subsets $\mathcal{A}, \mathcal{B}$ of $\bar\Omega$ such that $\bar{\mathcal{A}} \subset \hat{V}$ and $\bar{\mathcal{B}} \subset \hat{W}$, and such that $(\mu_x \otimes \mu_y)(\del(\mathcal{A} \times \mathcal{B})) = 0$.

Using Proposition \ref{prop:1ere_etage} and arguing as in \cite[Th.\,4.1.1]{Roblin} (Troisi\`eme \'etape), we have that
\begin{align*}
\limsup \nu_{x,y}^t(\mathcal{A} \times \mathcal{B}) & \leq e^\eps \mu_x(\mathcal{A}) \mu_y(\mathcal{B}) ,\\
\liminf \nu_{x,y}^t(\mathcal{A} \times \mathcal{B}) & \geq e^{-\eps} \mu_x(\mathcal{A}) \mu_y(\mathcal{B}) .
\end{align*} 
One may deduce from these inequalities that 
for all positive continuous functions $h$ supported on $\hat{V} \times \hat{W}$, we have
\[ e^{-\eps} \int \varphi \,d(\mu_x \otimes \mu_y) \leq \liminf \int \varphi \,d\nu_{x,y}^t \leq \limsup \int \varphi \,d\nu_{x,y}^t \leq e^\eps \int \varphi \,d(\mu_x \otimes \mu_y) .\]

Now, by compactness, there is a finite cover of $\bar\Omega \times \bar\Omega$ by open sets of the form $\hat{V} \times \hat{W}$. By using a partition of unity subordinate to this finite cover, we see that this last set of inequalities remains valid for all continuous positive functions on $\bar\Omega \times \bar\Omega$.

It then remains only to take $\eps \to 0$ to obtain the desired convergence of integrals. This conclude the proof.
\end{proof}

\section{Equidistribution of primitive closed geodesics} \label{sec:pcgeod_equidist}

In this section we prove an equidistribution result for primitive closed geodesics, again with counting results for such geodesics as a consequence:
\begin{thm}[cf. {\cite[Th.\,5.1.1 \& 5.2]{Roblin}}] \label{thm:pcgeod_equidist}
Suppose $\Omega$ is a strictly convex projective domain with $C^1$ boundary and $\Gamma \leq \Aut(\Omega)$ is a non-elementary discrete subgroup such that $S \Omega / \Gamma$ admits a finite Sullivan measure $m_\Gamma$ associated to a $\Gamma$-equivariant conformal density $\mu$ of dimension $\delta(\Gamma)$.

\begin{enumerate}[(a)]
\item As $\ell \to +\infty$,
\[ \delta \ell e^{-\delta \ell} \sum_{g \in \mathcal{G}_\Gamma(\ell)} \dirac_g \to \frac{m_\Gamma}{\|m_\Gamma\|} \]
weakly in $C_c(S\Omega/\Gamma)^*$;
\item If furthermore $\Gamma$ acts geometrically finitely on $\Omega$, we have the same convergence in $C_b(S\Omega / \Gamma)^*$.
\end{enumerate}
\end{thm}
Here $\mathcal{G}_\Gamma(\ell)$ denotes the set of primitive closed geodesics of length at most $\ell$, and for $g \in \mathcal{G}_\Gamma = \lim_{\ell\to\infty} \mathcal{G}_\Gamma(\ell)$ any primitive closed geodesic $\dirac_g$ denotes the normalized Lebesgue measure supported on $g$. $C_c(S\Omega / \Gamma)^*$ and $C_b(S\Omega / \Gamma)^*$ denote the weak* duals of, respectively, the space of compactly-supported continuous functions on $S\Omega / \Gamma$, and
the space of {\it bounded} continuous functions on $S\Omega / \Gamma$. 
% The proof of part (b) of this Theorem passes through / builds on the proof of part (a), and so we include both of these statements here.
% Part (b) is needed to obtain the following corollary in the general case where $S\Omega/\Gamma$ is finite-volume but not necessarily compact:
\begin{cor}[cf. {\cite[Cor.\,5.3]{Roblin}}]
$\#\mathcal{G}_\Gamma(\ell) \sim \frac{e^{\delta\ell}}{\delta\ell}$ as $\ell\to+\infty$.
\begin{proof}
Integrating the constant function 1 against the measure $\delta \ell e^{-\delta \ell} \sum_{g \in \mathcal{G}_\Gamma(\ell)} \dirac_g$ gives $\#\mathcal{G}_\Gamma(\ell)$. From Theorem \ref{thm:pcgeod_equidist}, this integral converges to 1 as $ \ell \to \infty$. 
\end{proof}
\end{cor}

As above, these results extend and are inspired by results first proven in the context of closed manifolds of constant negative curvature \cite{Margulis} and subsequently extended to much more general settings; we refer the interested reader to the beginning of \cite[Ch.\,5]{Roblin} for a more extended version of this history.
% As in the previous section, similar results, using a slightly different notion of length, are known to hold for convex cocompact strictly convex projective holonomies by the work of Sambarino in \cite{Sambarino_cvx}. 

\begin{proof}[Proof of Theorem \ref{thm:pcgeod_equidist}(a)]
Denote by $\Epsy^L$ the measure on $S\Omega$ which descends to $\delta L e^{-\delta L} \sum_{g \in \mathcal{G}_\Gamma(L)} \dirac_g$ on $S\Omega/\Gamma$. We need to prove that $\Epsy^L \to \frac{m}{\|m_\Gamma\|}$ weakly in $C_c(S\Omega)^*$ when $L\to+\infty$; we remind the reader here that $m$ is the measure on $S\Omega$ which descends to $m_\Gamma$ on $S\Omega/\Gamma$. 
% Let $\mu$ be the $\Gamma$-invariant conformal density of dimension $\delta$ to which $m$ is associated,
% \footnote{If $\Omega/\Gamma$ is convex co-compact, there is a unique one. If $\Omega/\Gamma$ is geometrically finite but not convex co-compact, there may not be, in which case take the Patterson-Sullivan density described in \S\ref{subsec:pat_sul}}
Abusing notation slightly, we write $\mu$ to also denote the measure on $\del^2\Omega$ given by
\[ d\mu(\xi,\eta) = e^{ \delta \cdot (\beta_\xi(x,u) + \beta_\eta(x,u))} \,d\mu_x(\xi) \,d\mu_x(\eta) = e^{2\delta\cdot \grop{x}\xi\eta} d\mu_x(\xi) d\mu_x(\eta) \]
which is independent of the choice of $x \in \Omega$ and $u \in (\xi\eta) \subset \Omega$. We recall again that by definition, $m = \mu \otimes ds$ on $S\Omega = \del^2\Omega \times \real$.

We will first use Theorem \ref{thm:orbit_equidist} to obtain a measure $\nu_{x,1}^L$ converging weakly to $\mu$ when $L\to+\infty$, then successively modify $\nu_{x,1}^L$ to form $\nu_{x,2}^L$ and $\nu_{x,3}^L$, so that $\nu_{x,3}^L$ will be supported on pairs of fixed points of hyperbolic elements, and that $\nu_{x,3}^L$ locally approaches $\mu$. By taking the product of $\|m_\Gamma\|^{-1} \nu_{x,3}^L$ with the Lebesgue measure on $\real$, we obtain a measure $\Em_{x,3}^L$ approaching $\|m_\Gamma\|^{-1} m$ locally (i.e. near the fibre over $x\in\Omega$ in $S\Omega$.) To finish, we relate $\Em_{x,3}^L$ 
to the measure of equidistribution $\Epsy^L$. 

Fix for now $x \in \Omega$ and $r > 0$, and let $V(x,r)$ denote the open set of pairs $(a,b) \in \Omega^2 \cup \del^2\Omega$ such that the geodesic $(a,b)$ intersects $B(x,r)$. By Theorem \ref{thm:orbit_equidist}, the measure
\[ \nu_{x,1}^L := \delta \|m_\Gamma\| e^{-\delta L} \sum_{d_\Omega(x,\gamma x) \leq L} \dirac_{\gamma^{-1}x} \otimes \dirac_{\gamma x} \]
converges weakly in $C(\bar\Omega \times \bar\Omega)^*$ to $\mu_x \otimes \mu_x$ as $L\to+\infty$. We restrict henceforth these measures to the open set $V(x,r)$, that is to say that we will consider them as elements of $C_c(V(x,r))^*$. Since $0 \leq \grop{x}\xi\eta \leq r$ for $(\xi,\eta) \in \del^2\Omega \cap V(x,r)$, we have, for all $\psi \in C_c^+(V(x,r))$, 
\[ e^{-2\delta r} \int \psi\,d\mu \leq \liminf \int \psi\,d\nu_{x,1}^L \leq \limsup \int \psi\,d\nu_{x,1}^L \leq \int \psi\,d\mu \]
as $L\to+\infty$.

Now as we start to modify our measure, we will make use of the following geometric lemma:
% s'appuyer = ``to lean on''
\begin{lem}[cf. {\cite[p.\,68]{Roblin}}] \label{lem:511}
Let $x \in \Omega$, and fix $r > 0$ and $\eps > 0$.

There exists $t_0 = t_0(x,r,\eps)$ such that if $\phi \in \Isom(\Omega,d_{\Omega})$ satisfies $d_\Omega(x, \phi x) > t_0$ and $(\phi^{-1} x , \phi x)\cap \bar{B}(x,r) \neq \varnothing$, then $\phi$ is hyperbolic, and $e^{-\grop{x}{\phi^{\pm1} x}{\phi^\pm}} < \eps$ where $\phi^\pm$ denote the attracting / repelling fixed points of $\phi$. 
\end{lem}

Informally: we can identify hyperbolic isometries $\phi$ by looking for translation-like behavior, and these exhibit North-South dynamics---for such $\phi$ and $x$ close enough to the axis of $\phi$ with $(x, \phi x)$ sufficiently large, $\phi^\pm x$ get arbitrarily close to $\phi^\pm$.)

\begin{proof}
Suppose $(\phi^{-1} x \, \phi x)\cap \bar{B}(x,r) \neq \varnothing$. 
If $\phi$ is parabolic (elliptic, resp.), then $x, \phi^{-1} x, \phi x$ are all situated on a single horosphere (a circle, which tends to a horosphere as the radius increases), and so $d_\Omega(x, \phi x)$ cannot be too large without violating the condition that $(\phi^{-1} x , \phi x)\cap \bar{B}(x,r) \neq \varnothing$. Thus, if the distance is larger than a certain $t_0$ (depending on $x, r, \eps$), then $\phi$ must necessarily be hyperbolic, and hence, in this case, is positive proximal, i.e. has a largest eigenvalue which is positive and of multiplicity one (see e.g. \cite[Prop.\,2.8]{CLT}). 

Now $e^{-\grop{x}{\phi^{\pm1}x}{\phi^\pm}} = e^{-\frac12 d_\Omega(x,\phi^{\pm1}x)} e^{\frac 12 \beta_{\phi^\pm}(x,\phi^{\pm1} x)} = \exp(-\frac 12 d_\Omega(x, \phi^{\pm1} x) - \frac 12\tau(\phi))$, where $\tau(\phi)$ denotes the translation distance of $\phi$. We note that this is bounded above by $e^{-\frac12 d_\Omega(x,\phi^\pm x)}$, which clearly tends to 0 as $d_\Omega(x, \phi x) = d_\Omega(\phi^{-1} x, x) \to \infty$. 
\end{proof} 
% We remark that this proof was in some ways slightly easier than the corresponding proof in Roblin, because in our setting horospheres have more structure than in the general CAT(-1) setting---see \cite{CLT}, \S3.

Write $\Gamma_h$ to denote the set of hyperbolic elements of $\Gamma$. According to Lemma \ref{lem:511}, if $\gamma\in\Gamma$ is elliptic or parabolic, i.e. if $\gamma\notin\Gamma_h$, then $(\gamma^{-1}x, \gamma x) \notin V(x,R)$ as soon as $d_\Omega(x, \gamma x)$ is large enough, i.e. for all but finitely many $\gamma$. As a consequence, if we define the measure
\[ \nu_{x,2}^L := \delta \|m_\Gamma\| e^{-\delta L} \sum_{\substack{\gamma\in\Gamma_h \\ d_\Omega(x,\gamma x) \leq L}} \dirac_{\gamma^{-1}x} \otimes \dirac_{\gamma x} \]
we have $\nu_{x,1}^L - \nu_{x,2}^L \to 0$ as $L\to+\infty$ (still restricting to $V(x,r)$.)

For $\gamma \in \Gamma_h$, write $\gamma^\pm$ for its attracting and repelling fixed points. By Lemma \ref{lem:511}, if $(\gamma^{-1} x, \gamma x) \in V(x,r)$, then $\gamma^\pm x$ are uniformly arbitrarily close to $\gamma^\pm$ as long as $d_\Omega(x,\gamma x)$ is large enough, which is to say for all but finitely many $\gamma$. Defining the measure
\[ \nu_{x,3}^L := \delta \|m_\Gamma\| e^{-\delta L} \sum_{\substack{\gamma\in\Gamma_h \\ d_\Omega(x,\gamma x) \leq L}} \dirac_{\gamma^-} \otimes \dirac_{\gamma^+} \]
we have $\nu_{x,3}^L - \nu_{x,2}^L \to 0$ weakly as $L\to+\infty$, when restricted to $V(x,r)$.

It now follows from the preceding series of inequalities and convergences that
\[ e^{-2\delta r} \int \psi\,d\mu \leq \liminf \int \psi\,d\nu_{x,3}^L \leq \limsup \int \psi\,d\nu_{x,3}^L \leq \int \psi\,d\mu \]
for all $\psi \in C_c^+(V(x,r))$, and hence for $\psi \in C_c^+(\del^2\Omega \cap V(x,r))$, since the measures we are talking about are supported on $\del^2\Omega$. 

For $\gamma \in \Gamma_h$, we let $g_\gamma \subset S\Omega$ denote the oriented axis of $\gamma$, and $\Leb_\gamma$ denote the Lebesgue measure supported on $g_\gamma$, and finally
\[ \Em_{x,3}^L = \delta e^{-\delta L} \sum_{\substack{\gamma\in\Gamma_h \\ d_\Omega(x,\gamma x) \leq L}} \Leb_\gamma .\]
In other words, $\Em_{x,3}^L = \|m_\Gamma\|^{-1} \nu_{x,3}^L \otimes ds$. Let $\hat{V}(x,r) := V(x,r) \times \real \subset S\Omega$. From the preceding, we obtain, for $\psi \in C_c^+(\hat{V}(x,r))$, 
\begin{align} 
e^{-2\delta r} \|m_\Gamma\|^{-1} \int \psi\,dm & \leq \liminf \int \psi\,d\Em_{x,3}^L \notag \\ & \leq \limsup \int \psi\,d\Em_{x,3}^L \leq \|m_\Gamma\|^{-1} \int \psi\,dm \label{eqn:511_Emx3t} 
\end{align}

Then, arguing as in the proof of \cite[Th.\,5.1.1]{Roblin}, we may establish that for all $\varphi \in C_c^+(\hat{V}(x,r))$, as $t\to+\infty$,
\begin{align*} 
\resizebox{\linewidth}{!}{$\liminf \int \varphi\,d\Em_{x,3}^L \leq \liminf \int \varphi\,d\Epsy^L
  \leq \limsup \int \varphi\,d\Epsy^L \leq e^{(2\delta+1)r} \limsup \int \varphi\,d\Em_{x,3}^L$}
\end{align*}
so that
\begin{align*}
e^{-2\delta r} \|m_\Gamma\|^{-1} \int \varphi\,dm & \leq \liminf \int \varphi\,d\Epsy^L \\ 
 & \leq \limsup \int \varphi\,d\Epsy^L \leq e^{(2\delta+1)r} \|m_\Gamma\|^{-1} \int \varphi\,dm .
\end{align*} 

Appealing to a locally-finite partition of unity subordinate to a covering of $S\Omega$ by $\hat{V}(x,r)$ with $x \in \Omega$ and $r > 0$ fixed, we extend the validity of the preceding inequalities to all functions $\varphi \in C_c^+(S\Omega)$. It remains only to take $r \to 0$ to conclude the proof. 
\end{proof}

\begin{proof}[Proof of Theorem \ref{thm:pcgeod_equidist}(b)]
Let $\Epsy_\Gamma^L$ denote the measure $\delta L e^{\delta L} \sum_{g \in \mathcal{G}_\Gamma(L)} \dirac_g$ on $S\Omega/\Gamma$. By part (a), we already know that $\Epsy_\Gamma^L \to \frac{m_\Gamma}{\|m_\Gamma\|}$ weakly in $C_c(S\Omega/\Gamma)^*$ when $L\to+\infty$. We start by replacing $\Epsy_\Gamma^L$ by a nearby measure which will be better adapted to the argument to come, namely
\[ \Em_\Gamma^L := \delta e^{-\delta L} \sum_{g\in\mathcal{G}_\Gamma(L)} \ell(g) \dirac_g \]
(we remark that $\ell(g)\dirac_g$ is simply the non-normalized Lebesgue measure along $g$.) 

We may verify, by arguing as in the proof of \cite[Th.\,5.2]{Roblin}, that we still have $\Em_\Gamma^L \to \frac{m_\Gamma}{\|m_\Gamma\|}$ weakly in $C_c(S\Omega/\Gamma)^*$ when 
$L\to+\infty$. and that it suffices to show that 
$\Em_\Gamma^L$ converges weakly to $\frac{m_\Gamma}{\|m_\Gamma\|}$ in $C_b(S\Omega/\Gamma)^*$ as $L\to+\infty$, to obtain the same (desired) conclusion for $\Epsy_\Gamma^L$.

{\bf The rest of the proof consists in demonstrating that $\Em_\Gamma^L$ converges weakly to $\frac{m_\Gamma}{\|m_\Gamma\|}$ in $C_b(S\Omega/\Gamma)^*$ as $L\to+\infty$.} We present this step in more detail since it more intimately involves the Hilbert geometry in the cusps.

\newcommand{\fundom}{F}
\newcommand{\perifundom}{D}
Per Theorem \ref{thm:cvx_lf_fundoms}, choose a convex locally-finite fundamental domain $\fundom \subset \Omega$ for the action of $\Gamma$ on $\Omega$ with $S^*m$-negligible boundary, where $S^*m$ is the measure on $\Omega$ given by $S^*m(B) := m(SB)$. Then $S\fundom$ is a locally-finite fundamental domain for the action of $\Gamma$ on $S\Omega$, with $m$-negligible boundary.

Take $\mathcal{P}$ to be a system of representatives of $\Gamma$-orbits of parabolic fixed points of $\Gamma \actson \bar\Omega$, and, for $\xi_P \in \mathcal{P}$, let $P = \Stab_\Gamma(\xi_P)$. As in the last bit of \S1.4, $\mathcal{P}$ is finite and for each $\xi_P \in \mathcal{P}$ we can find a horoball $H_P$ based at $\xi_P$ such that $\gamma H_P \cap H_P \neq \varnothing$ if and only if $\gamma \in P$. Given $r > 0$, we let $H_P(r)$ denote the horoball contained in $H_P$ whose boundary is distance $r$ from that of $H_P$. Then, writing
\[ \fundom_r := \fundom \setminus \bigcup_{\xi_P \in \mathcal{P}} \Gamma H_P(r) \]
to denote the ``thick part'' of the fundamental domain, we know that the intersection of $\fundom_r$ with the convex hull of the limit set $\Lambda_\Gamma$ is compact in $\Omega$.

Let $\Em^L$ denote the measure on $S\Omega$ which projects to $\Em_\Gamma^L$, and $S^*\Em^L$ denote the measure on $\Omega$ defined by $S^*\Em^L(B) := \Em^L(SB)$. Since $\Em_\Gamma^L$ converges weakly to $\frac{m_\Gamma}{\|m_\Gamma\|}$ in $C_c(S\Omega/\Gamma)^*$ as $L\to+\infty$, and the intersection of $S\fundom_r$ with the support of $m$ (which also contains the support of $\Em^L$) is compact for each $r \geq 0$, it suffices to show that
\[ \limsup_{L\to+\infty} S^*\Em^L(\fundom \setminus \fundom_r) \to 0 \]
as $r \to +\infty$, or that, for each $\xi_P \in \mathcal{P}$,
\[ \limsup_{L\to+\infty} S^*\Em^L(\fundom \cap \Gamma H_P(r)) \to 0 \]
Informally: convergence in $C_c$ gives us control over the thick part, so {\bf it suffices to  control what happens in each of the finitely many cusps.}

The remainder of the argument will resemble a more refined version of the argument in the proof of Theorem \ref{thm:finite_BMmeas}: whereas there we had a finite bound for the measure of the cusps, here we want a bound that asymptotically goes to zero. 

Fix, from here on, $\xi_P \in \mathcal{P}$. We will omit, in the sequel, the index $P$ (so $H = H_P$, $\xi = \xi_P$, etc.) in the interest of brevity. Choose a compact fundamental domain  $\perifundom$ for the action of $\Gamma$ on $\Lambda_\Gamma \smallsetminus \{\xi\}$. Modifying if necessary the locally-finite fundamental domain $\fundom$ for $\Gamma$ which we had previously chosen, we may assume that $\fundom \cap \Gamma H = \fundom \cap H$, and hence that $\fundom \cap \Gamma H(r) = \fundom \cap H(r)$ for $r \geq 0$. 

Given a hyperbolic isometry $\gamma \in \Gamma$, we let $g_\gamma \subset \Omega$ denote its unoriented axis, and $\Leb_\gamma$ the Lebesgue measure along $g_\gamma$. We have $S^*\Em^L = \delta e^{-\delta L} \sum \ell(\gamma)$, where the sum is taken over $\gamma \in \Gamma_{hp}$ with $\ell(\gamma) \leq L$. Since the endpoints of $g_\gamma$ are in $\Lambda_\Gamma \setminus \{p\} = \bigcup_{\pi \in P} \pi \perifundom$, we may write
\[ S^*\Em^L(\fundom \cap H(r)) = \delta e^{-\delta L} \sum_{\pi_1 \in P} \sum_{\pi_2 \in P} \sum \Leb_\gamma(\fundom \cap H(r)) \]
where the third sum is taken over $\gamma \in \Gamma_{hp}$ with $\ell(\gamma) \leq L$ and with axis going from $\pi_1 \perifundom$ to $\pi_2 \perifundom$. As on \cite{Roblin}, p. 74, 
writing $\Gamma(L, \pi)$ to denote the set of $\gamma \in \Gamma_{hp}$ with $\ell(\gamma) \leq L$ and with axis going from $\perifundom$ to $\pi \perifundom$, 
we may rewrite this as
\begin{equation} 
S^*\Em^L(\fundom \cap H(r)) = \delta e^{-\delta L} \sum_{\pi \in P} \sum_{\gamma \in \Gamma(L,\pi)} \Leb_\gamma(H(r)) 
\label{eqn:52_decomp_horos} \end{equation}

{\bf We now bound from above the quantity $\sum_{\Gamma(L,\pi)} \Leb_\gamma(H(r))$.} Choose a geodesic going from $\perifundom$ to $\xi$, and denote by $x$ the point of intersection of this geodesic and the horosphere $\del H$. We will use the following geometric lemma, which uses the strict convexity of the horoballs:

\begin{lem} \label{lem:52}
Suppose we are given a compact fundamental domain $\perifundom$ for $P \actson \Lambda_\Gamma \smallsetminus \{\xi\}$, a geodesic $g_0$ from $\perifundom$ to $\xi$, and a horoball $H$. Write $x := g_0 \cap \del H$.

Then there exists $\kappa > 0$ such that if a geodesic $g$ from $\perifundom$ enters the closed horoball $\bar{H}$ at a point $y \in \del H$, then $d_\Omega(x,y) \leq \frac12 \kappa$.
\begin{proof}
We note that $F \subset \shadow_{r'}(\xi,x)$ for some $r' > 0$ that depends only on $P$ and $F$. By the relative compactness of $\perifundom$, we may take $r'$ independent of $g$. Since $\mathcal{P}$ is finite, there is a single $r'$ which works for all $\xi \in \mathcal{P}$. 
We can then take $\kappa = 2r'$. 
\end{proof}
\end{lem}
% We remark that this lemma can also be seen to follow from hyperbolic case, plus the projective equivalence of geometrically finite cusps to hyperbolic ones 
% (\cite{CM12}, Th\'eor\`eme 1.7.)
% (\cite{CLT}, Theorem 0.5.)

Now consider $\gamma \in \Gamma(L,\pi)$ such that $\Leb_\gamma(H(r)) > 0$, i.e. the axis $g_\gamma$ intersects $H(r)$. Let $a$ be the point where $g_\gamma$ enters $\bar{H}$, $b$ the point where it enters $\overline{H(r)}$, $c$ the point where it exits $\overline{H(r)}$ and $d$ the point where it exits $\bar{H}$. By Lemma \ref{lem:52}, we have on the one hand $d_\Omega(a,x) \leq \frac 12 \kappa$, and on the other hand $d_\Omega(d,\pi x) = d_\Omega(\pi^{-1}d,x) \leq \frac12 \kappa$ since the axis $\pi^{-1} g_\gamma$ runs, in the opposite direction, from $\perifundom$, entering $\pi^{-1}\bar{H} = \bar{H}$ at the point $\pi^{-1}d$. Moreover, $d_\Omega(a,b) \geq r$ and $d_\Omega(c,d) \geq r$. It then follows that
\[ \Leb_\gamma(H(r)) = d_\Omega(b,c) = d_\Omega(a,d) - d_\Omega(a,b) - d_\Omega(c,d) \leq d_\Omega(x,\pi x) + \kappa - 2r \]
Note, in particular, that $d_\Omega(x,\pi x) \geq 2r - \kappa$.
Moreover, $d_\Omega(x, \gamma x) \leq d_\Omega(a, \gamma a) + \kappa = \ell(\gamma) + \kappa \leq L + \kappa$. We will need these facts below, soon.

We will finish by bounding from above the cardinality of the set $\Gamma(L,\pi,r)$ of $\gamma \in \Gamma(L,\pi)$ such that $\Leb_\gamma(H(r)) > 0$, by applying the shadow lemma (Lemma \ref{lem:sullivan_shadow}) to the $\Gamma$-equivariant conformal density $\mu$ of dimension $\delta$ associated to $m$. 

Take $R_0 > \kappa$ large enough, depending only on $\Gamma$, so that the shadow lemma applies to the shadow of the balls $B(\gamma x, R)$ viewed from $x$ for all $R \geq R_0$. 
Since the axis $g_\gamma$ is at a distance of at most $\frac12 \kappa \leq \frac12 R$ from $x$ as well as from $\gamma x$, any shadow $\shadow_\gamma$ of $B(\gamma x, R)$ viewed from $x$ contains the attracting fixed point of $\gamma$, and thus intersects $\pi \perifundom$. 
$\pi^{-1}\shadow_\gamma$ is the shadow of $B(\pi^{-1}\gamma x, R)$ viewed from $\pi^{-1}x$ and intersects $\perifundom$ (indeed, contains the repelling fixed point of $\gamma$.)

Because $\perifundom$ is compact, we can choose $R \geq R_0$, depending only on $\Gamma$ and $P$, such that there exists some compact set $K \subset \del\Omega \smallsetminus \{\xi\}$ containing all of the $\pi^{-1} \shadow_\gamma$ for $\gamma \in \Gamma(L,\pi,r)$. To ensure that $R$ is not so large that these shadows end up being all of $\del\Omega$, it is helpful to require e.g. $R < 2r$; for the shadow lemma to still apply we would then need $r > \frac\kappa2$, but this is fine since we are taking $r\to+\infty$. Fix this value of $R$, and  let $C = C_{x,R}$ be the constant from the shadow lemma.

Given $\gamma \in \Gamma(L,\pi,r)$, we have $\shadow_\gamma \subset \pi K$. Now for all $t > 0$, the family 
\[ \{\shadow_\gamma : \gamma\in\Gamma(L,\pi,r), t-1 < d_\Omega(x,\gamma x) \leq t \} \] 
forms a cover of some open subset of $\pi K$ of uniformly bounded multiplicity $M$, by Lemma \ref{lem:roblin_1B}. Since $\mu_x(\shadow_\gamma) \geq C^{-1} e^{-\delta t}$ by the shadow lemma, it follows that 
\[ \# \{ \gamma\in\Gamma(L,\pi,r) \st t-1 < d_\Omega(x,\gamma x) \leq t \} \leq CM e^{\delta t} \mu_x(\pi K) .\] 

We have seen above that $d_\Omega(x,\gamma x) \leq L + \kappa$ for all $\gamma \in \Gamma(L,\pi,r)$; by summing the previous estimate over positive integers $t$, we have that the cardinality of $\Gamma(L,\pi,r)$ is bounded above by $\frac{CM e^{\delta\kappa}}{e^\delta-1} e^{\delta L} \mu_x(\pi K)$.

It hence suffices to bound $\mu_x(\pi K)$ from above. We have
\[ \mu_x(\pi K) = \mu_{\pi^{-1} x}(K) = \int_K e^{-\delta \beta_\xi(\pi^{-1} x, x)} \,d\mu_x(\xi) .\]
Now observe that we can choose $r_0 > 0$ depending only on $P$ such that $K$ is contained in the shadow, viewed from $\pi x$, of a ball with center $x$ and radius $r_0$. An application of Lemma \ref{lem:roblin_12} then shows that $\mu_x(\pi K)$ is bounded above by $e^{2\delta r_0} e^{-\delta\cdot d_\Omega(x, \pi x)}$.

Putting these together, we obtain, ultimately, that 
\[ \#\Gamma(L,\pi,r) \leq \frac{CM e^{\delta(\kappa+2r_0)}}{e^\delta - 1} e^{\delta L - \delta\cdot d_\Omega(x,\pi x)} .\]

Combining this with (\ref{eqn:52_decomp_horos}) yields 
\[ S^*\Em^L(\fundom \cap H(r)) \leq \hat{C} \sum_{\substack{\pi\in P \\ d_\Omega(x,\pi x) > 2r-\kappa}} (d_\Omega(x,\pi x) - 2r + \kappa) e^{-\delta\cdot d_\Omega(x,\pi x)} \]
where $\hat{C} = \frac{\delta CMe^{\delta(\kappa+2r_0)}}{e^\delta-1}$ is a constant independent of $L$ and $r$. 

We now observe that, because $\delta_\Gamma > \delta_P$ (Lemma \ref{lem:critgap_para}), 
$\sum_{\pi\in P} d_\Omega(x,\pi x) e^{-\delta\cdot d_\Omega(x,\pi x)}$ converges.
It then follows from the convergence of this series that
\[ \limsup_{L\to+\infty} S^*\Em^L(\fundom \cap H(r)) \to 0 \]
as $r\to+\infty$. Thus we have controlled the measure in the cusps, and as described above this concludes the proof.
\end{proof}